\documentclass[11pt]{amsart}

\usepackage{amssymb}
\usepackage{overpic}
\usepackage{enumitem}   
\usepackage{graphicx} 
\usepackage{amsrefs}
\usepackage{xcolor}
\usepackage{tikz-cd} 
\usepackage[hidelinks]{hyperref}

\graphicspath{{Figures/}} 

\newtheorem{theorem}{Theorem} 
\newtheorem{proposition}{Proposition} 
 
\newtheorem{remark}{Remark} 
 
\newtheorem{definition}{Definition} 
\newtheorem{example}{Example}

\newtheorem{main}{Theorem} 

\begin{document}
	
\title[A Piecewise smooth $\lambda$-Lemma]{A Piecewise smooth \boldmath{$\lambda$}-Lemma}

\author[Claudio Buzzi, Paulo Santana, and Luan V. M. F. Silva ]{Claudio Buzzi$^1$, Paulo Santana$^1$, and Luan V. M. F. Silva$^1$}

\address{$^1$ IBILCE--UNESP, CEP 15054--000, S. J. Rio Preto, S\~ao Paulo, Brazil}
\email{claudio.buzzi@unesp.br}
\email{luan.m.silva@unesp.br}
\email{paulo.santana@unesp.br}



\begin{abstract}
	In this paper we provide extensions of the $\lambda$-Lemma (also known as Inclination Lemma) for piecewise smooth vector fields and maps. In order to achieve our main result, we investigate the regularity of time-$T$-maps of piecewise smooth vector fields defined at crossing orbits. We prove that these maps are homeomorphisms and also piecewise smooth diffeomorphisms.
\end{abstract}

\maketitle

\section{Introduction}\label{Sec1}
The understanding of a dynamical system often requires identifying the existence of invariant sets in the phase space, such as fixed points, singularities, invariant tori, and periodic orbits. The presence of chaotic sets is equally important in such analysis, as it directly influences the system's predictability. A key tool for studying this phenomenon is determining whether the dynamical system exhibits a hyperbolic structure. In such cases, the classical theory asserts that if there is a transversal intersection between the stable and unstable manifolds of a fixed point (in the case of maps) or a periodic orbit (for vector fields in dimension greater than two), then chaos is present. This result is formalized in the Birkhoff-Smale theorem~\cite{Smale}. Nowadays one of the keystones in the proof the Birkhoff-Smale theorem (see~\cite[Theorem~$5.2.10$]{GukHol1983} and~\cite[Proposition~$2.5$]{New}) is the theorem known as \emph{$\lambda$-Lemma}~\cite{Palis} (and also as \emph{Inclination Lemma}~\cite{Palis2}). 

When dealing with piecewise smooth maps or piecewise smooth vector fields (PSVF for short), determining the existence of chaotic sets through this classical method becomes more difficult. The lack of smoothness complicates the application of traditional (smooth) tools in such analysis.

Therefore the main goal of this paper is to provide an extension of the $\lambda$-Lemma for Poincar\'e maps of PSVF, in particular for those following the \emph{Filippov convention}~\cite{Filippov}, also known as \emph{Filippov vector fields}.

In order to clarify the objects and concepts discussed here, we first introduce them in the classical context before adapting it to our specific setting.

Let $M$ be a compact smooth manifold (i.e. of class $\mathcal{C}^\infty$) of dimension $m$ and without boundary. Let also $F\colon M\to M$ be a $\mathcal{C}^1$-diffeomorphism.

A fixed point $p\in M$ is \emph{hyperbolic} if for every eigenvalue $\alpha\in\mathbb{C}$ of $DFn(p)$ we have $|\alpha|\neq1$. Let $d$ be a Riemmanian metric on $M$. It follows from the Stable Manifold Theorem for diffeomorphisms~\cite{HirPugh} that the \emph{stable and unstable manifolds} 
	\[W^s(x)=\{y\in M\colon d(F^n(x),F^n(y))\to0 \text{ as } n\to\infty\},\]
	\[W^u(x)=\{y\in M\colon d(F^{-n}(x),F^{-n}(y))\to0 \text{ as } n\to\infty\},\]
are well defined immersed submanifolds of $M$. Moreover, if $s=\dim W^s(p)$ and $u=\dim W^u(p)$, then $s+u=m$. We say that $p$ is a \emph{saddle} if $s\geqslant1$ and $u\geqslant1$.

Given $r\in\{1,\dots,m\}$, we say that $D\subset M$ is a \emph{$r$-disk} if it is diffeomorphic to a closed ball of $\mathbb{R}^r$. Given $p\in M$, we say that $D\subset M$ is a $r$-neighborhood of $p$ if $D$ is diffeomophic to an open ball of $\mathbb{R}^r$ and $p\in D$. 

In his seminal work on Morse-Smale dynamical systems, Palis~\cite{Palis} introduced a technical lemma that played a key role in the main results of his paper. Today, this technical result is considered foundational in its own right and it is known as \emph{$\lambda$-Lemma}.
\begin{theorem}[$\lambda$-lemma~\cite{Palis}]
	Let $p$ be a hyperbolic saddle of a $\mathcal{C}^1$-diffeomor\-phism $F\colon M\to M$ and $D\subset W^u(p)$ be a $u$-disk. Let $\Delta\subset M$ be a $u$-disk intersecting $W^s(p)$ transversely at some point $q$. Then $\cup_{n=0}^{\infty}F^n(\Delta)$ contains $u$-disks arbitrarily $\mathcal{C}^1$-close to $D$.
\end{theorem}

It follows from the $\lambda$-Lemma that if $\Delta\subset M$ is a sub-manifold of dimension at least $u$ intersecting $W^s(p)$ transversely, then 	
	\[
	W^u(p)\subset\overline{\cup_{n=0}^{\infty}F^n(\Delta)},
	\] 
where $\overline{A}$ denotes the topological closure of the set $A$. Thus, the iterates of $\Delta$ under $F$ are dragged to $p$ (due to the intersection with $W^s(p)$) and at the same time they must approach the entire unstable manifold in the $\mathcal{C}^1$-topology. As a consequence of this result, if $W\subset M$ is an invariant sub-manifold such that $\Delta\subset W$, then $W$ accumulates at the unstable manifold of $p$ (see Figure~\ref{Fig6}). In particular, if $W=W^u(p)$ (resp. $W=W^u(r)$ for some other hyperbolic saddle $r\in M$), then the dynamical system features a \emph{homoclinic} (resp. \emph{heteroclinic}) \emph{tangle} (see Figure~\ref{Fig7}).
\begin{figure}[ht]
	\begin{center}
		\begin{overpic}[width=6cm]{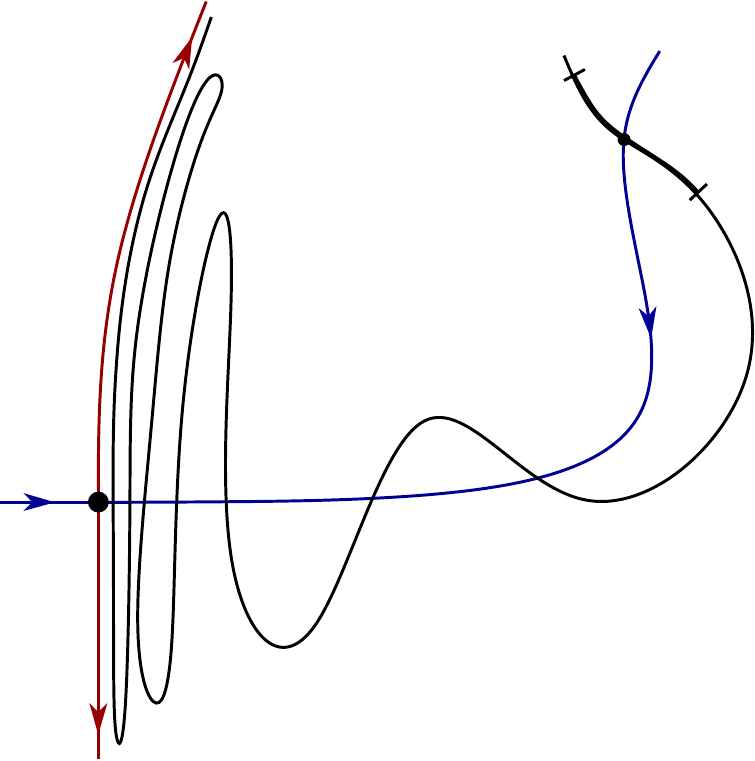} 
			\put(9,29){$p$}
			\put(78,78){$q$}
			\put(92,76.5){$\Delta$}
			\put(67,51){$W^s(p)$}
			\put(-2,70){$W^u(p)$}
			\put(43,15){$W$}
		\end{overpic}
	\end{center}
	\caption{Illustration of the $\lambda$-Lemma. The stable (resp. unstable) manifold of $p$ is represented in blue (resp. red). Colors available in the online version.}\label{Fig6}
\end{figure} 

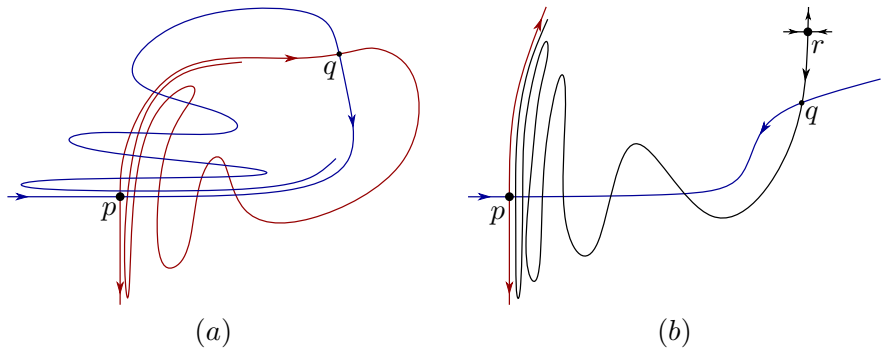
\begin{figure}[ht]
	\begin{center}
		\begin{minipage}{6cm}
			\begin{center} 
				\begin{overpic}[width=5.5cm]{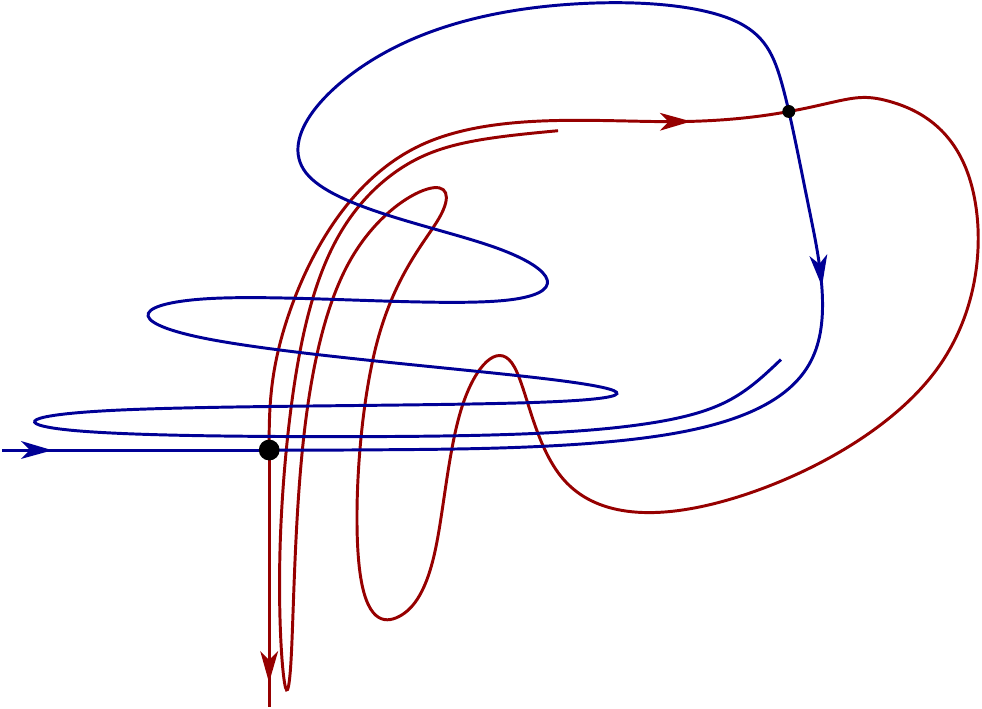} 
					\put(23,22){$p$}
					\put(76.5,56){$q$}
				\end{overpic}
				
				$(a)$
			\end{center}
		\end{minipage}
		\begin{minipage}{6cm}
			\begin{center} 
				\begin{overpic}[width=5.5cm]{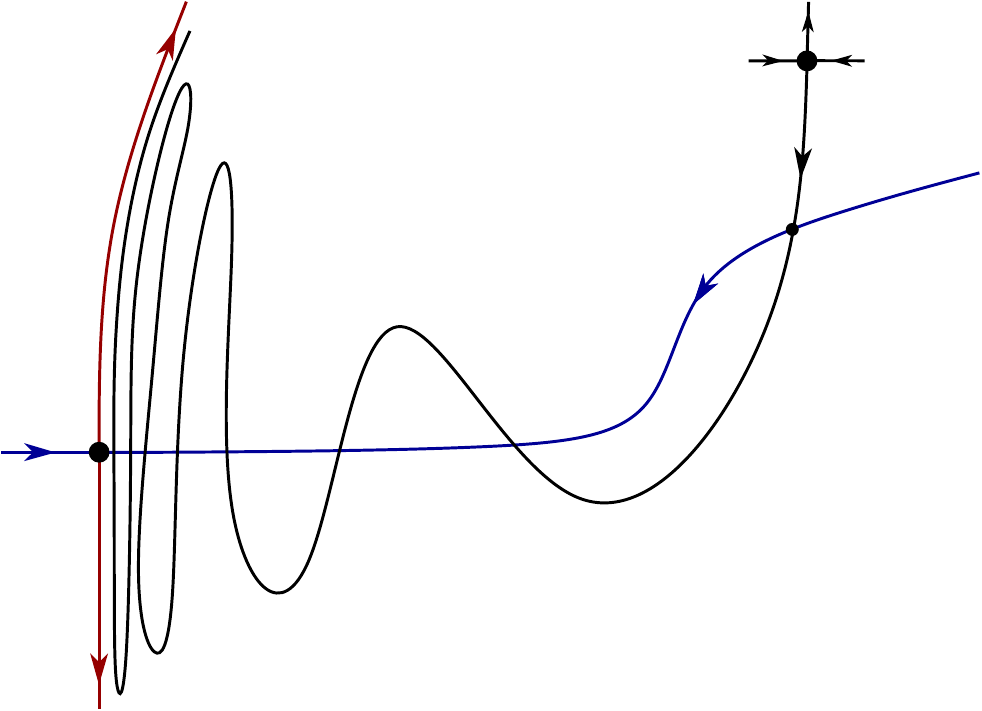} 
				 	\put(5.5,21){$p$}
				 	\put(81.5,45){$q$}
				 	\put(83,61){$r$}
				\end{overpic}
			
				$(b)$
			\end{center}
		\end{minipage}
	\end{center}
	\caption{Illustration of $(a)$ homoclinic $(b)$ heteroclinic tangles. The stable (resp. unstable) manifold of $p$ is represented in blue (resp. red). Colors available in the online version.}\label{Fig7}
\end{figure}

Besides the landmark results of Palis~\cite{Palis} about structural stability of diffeomorphisms and vector fields for $m\leqslant3$, the $\lambda$-Lemma was also applied to obtain simplified proofs of the Grobman-Hartman Theorem (see~\cite[Section~$2.7$]{PalisBook} for vector fields and \cite[Theorem~$2.1$]{Palis} for diffeomorphisms) and the Birkoff-Smale Homoclinic Theorem, see~\cite[Theorem~$5.3.5$]{GukHol1983}. 

Due to its importance the $\lambda$-Lemma was also extended in many ways, including versions for partially hyperbolic tori~\cite{Cre1}, normally hyperbolic invariant manifolds~\cite{Cre2}, locally Lipschitz vector fields~\cite{Pir}, some types of non-hyperbolic points~\cite{Vel}, heat equation~\cite{Web}, and the restrict planar three~\cite{Tere2} and four~\cite{Tere3} body problems.

The paper is organized as follows. Section~\ref{Sec2} provides the main result of the paper, Theorem~\ref{Main1}, dealing with the homoclinic tangle arising from a perturbation of an autonomous piecewise smooth vector field. Section~\ref{Sec3} is devoted to present some preliminaries results about PSVF and their originating Poincar\'e maps. Theorem~\ref{Main1} is proved in Section~\ref{Sec4}. Section~\ref{Sec5} offers further thoughts on generalizations of Theorem~\ref{Main1}, such as the heteroclinic tangle, proved at Theorem~\ref{Main3}, and also the non-perturbative framework, proved at Theorem~\ref{Main4}. In addition to that, it is discussed how the fact that the Poincaré map is \emph{embedded} in the flow of a piecewise smooth vector field is essential for both the statement and the proof of Theorem~\ref{Main1}. As a consequence, we present another version of Theorem~\ref{Main1}, namely Theorem~\ref{Main2}, which does not depend on such embedding.

\section{Statement of the main result}\label{Sec2}

In order to simplify the presentation of the main result, this section focuses on the case of homoclinic entanglement (as illustrated in Figure~\ref{Fig7}$(a)$). The heteroclinic entanglement and the more general cases (depicted in Figu\-res~\ref{Fig7}$(b)$ and \ref{Fig6}, respectively) are addressed in Section~\ref{Sec5}.

Let $h_i\colon\mathbb{R}^n\to\mathbb{R}$, $n\geqslant2$, $i\in\{1,\dots,N\}$, $N\geqslant1$, be smooth functions and set $\Sigma_i=h^{-1}(\{0\})$, where $0$ is not necessarily a regular value. We define the \emph{switching set} $\Sigma=\cup_{i=1}^{N}\Sigma_i$ and let $A_1,\dots, A_k$, $k\geqslant 2$, be the connected components of $\mathbb{R}^n\setminus\Sigma$. For each $j\in\{1,\dots,k\}$ let $X_0^j=X_0^j(x)$ be an autonomous smooth vector fields defined in a neighborhood of $\overline{A_j}$. We say that $X_0^j$ is a \emph{component} of the autonomous piecewise smooth vector field $Z_0:=(X_0^1,\dots,X_0^k)$.

Let $x_0\in\Sigma_i$ be such that $\nabla h_i(x_0)\neq0$ and let us consider $X$ as one of the components of $Z_0$ defined at $x_0$. The \emph{Lie derivative} of $h_i$ in the direction of the vector field $X$ at $x_0$ is given by
	\[Xh_i(x_0):=\left<X(x_0),\nabla h_i(x_0)\right>,\]
where $\left<\cdot,\cdot \right>$ denotes the standard inner product of $\mathbb{R}^n$.

\begin{definition}[Crossing point]\label{Def1}
	Let $Z_0$ be an autonomous piecewise smooth vector field with switching set $\Sigma$. We say that $x_0\in\Sigma$ is a crossing point of $Z_0$ if the following conditions are satisfied:
	\begin{enumerate}[label=(\roman*)]
		\item There is a unique $i\in\{1,\dots,N\}$ such that $x_0\in\Sigma_i$;
		\item $\nabla h_i(x_0)\neq0$;
		\item $X_0^ah_i(x_0)X_0^bh_i(x_0)>0$, where $X_0^a$ and $X_0^b$ are the only two components of $Z_0$ defined at $x_0$.
	\end{enumerate}
\end{definition}

Given a point $x_0\in\mathbb{R}^n$, the dynamics of $Z_0$ are as follows. If $x_0\not\in\Sigma$, then there is a unique $j\in\{1,\dots,N\}$ such that $x_0\in A_j$. Thus, the local trajectory of $Z_0$ at $x_0$ is given by the local trajectory of $X_0^j$, in the classical sense. If $x_0\in\Sigma$ is a crossing point, then it follows from Definition~\ref{Def1} that there are exactly two components, $X_0^a$ and $X_0^b$, of $Z_0$ defined at $x_0$ and that their local trajectories are transversal to $\Sigma$ and agree on orientation. Hence, the local trajectory of $Z_0$ at $x_0$ is defined as the concatenation of the local trajectories of $X_0^a$ and $X_0^b$, following the given orientation, which is determined by $\mbox{sign}(X_0^ah_i(x_0))=\mbox{sign}(X_0^bh_i(x_0))$. 

If $x_0\in\Sigma$ is not a crossing point, that is, if one of the conditions \emph{(i)}, \emph{(ii)} or \emph{(iii)} from Defintion~\ref{Def1} is not satisfied, then other notions of discontinuous points may be taken into account, such as \emph{sliding} and \emph{tangency points}. In such cases, local trajectories of $Z_0$ at $x_0$ are defined but they may not be unique. 

In this paper, roughly speaking, $Z_0$ is allowed to have non-crossing points. But for our purposes every arc of orbit (i.e. an orbit restricted to a compact interval) that we shall work with is supposed to intersect $\Sigma$ only in crossing points. Therefore if $x\in\mathbb{R}^n$ is such that its trajectory intersects $\Sigma$, then we shall assume that there is a compact interval of time such that it does only in crossing points. This enables us to define a global solution $\varphi_0(t;x)$ of $Z_0$, with initial condition $\varphi_0(0;x)=x$, given by the concatenation (in both forward and backward time) of the pieces of solutions of its components $(X_0^1,\dots,X_0^k)$. The arc of orbit of $Z_0$ through $x$ is denoted by $\gamma_x:=\{\varphi_0(t;x):t\in I_x\}$, where $I_x\subset\mathbb{R}$ is a compact interval containing $0$ in its interior and such that $\varphi_0(\cdot;x)$ is well defined for $t\in I_x$ and whenever it crosses $\Sigma$, it does in crossing points.

For the sake of the interested reader, for PSVF with non-crossing points, we refer to~\cites{SanDur2023,SilNun2019} for the cases that does not satisfy conditions $(i)$ or $(ii)$, and to~\cite{CarNovTon2024} for those that does not satisfy condition $(iii)$ of Definition~\ref{Def1}. Since non-crossing points also imply non-uniqueness of solutions, we also refer to~\cite{GomMatVar2023} for a study on the space of orbits, where uniqueness can be restored under particular conditions. For a general approach on defining the dynamics near non-crossing points, see~\cite{PanSil2017}. It is worthy mentioning that when the discontinuity set of the differential system is a regular manifold, the task of describing its solutions becomes significantly simplified, as observed by Filippov~\cite{Filippov} and Guardia et al~\cite{Marcel}. 

Let $p_0\in\mathbb{R}^n\setminus\Sigma$ be a hyperbolic saddle (i.e. a hyperbolic singularity of saddle type) of $Z_0$. Let $W_{loc}^s(p_0)$ and $W_{loc}^u(p_0)$ be the local stable and unstable manifolds of $p_0$, respectively. Since, in general, PSVF does not have uniqueness of solutions, we define the \emph{global} stable and unstable sets $W^{s,u}(p_0)$ as the \emph{saturation} of $W_{loc}^{s,u}(p_0)$ by the flow of $Z_0$, meaning that $W^{s,u}(p_0)$ is the union of all orbits of $Z_0$ having their initial conditions in $W_{loc}^{s,u}(p_0)$.

Now let $q_0\in W^s(p_0)\cap W^u(p_0)\setminus\{p_0\}$ be a homoclinic point of $Z_0$, associa\-ted with a homoclinic orbit $\gamma_{q_0}$, which intersects $\Sigma$ only in crossing points. This implies that the set $\Gamma=\{p_0\}\cup\gamma_{q_0}$ is topologically a circle and its intersection with $\Sigma$ consists solely of crossing points (see Figure~\ref{Fig8}).
\begin{figure}[ht]
	\begin{center}
		\begin{overpic}[width=6cm]{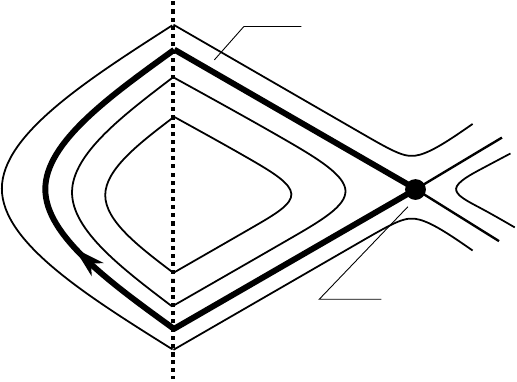} 
			\put(75.5,14.5){$p_0$}
			\put(60,66.5){$\Gamma$}
			\put(35,0){$\Sigma$}
		\end{overpic}
	\end{center}
	\caption{Illustration of the homoclinic connection $\Gamma$ of $Z$, in the case $n=2$.}\label{Fig8}
\end{figure} 

For each smooth component $X_0^j$ of $Z_0$, we consider the vector field $\mathcal{X}_0^j$ given by the product flow 
\begin{equation}\label{12}
	\dot x= X_0^j(x), \quad \dot t=1,
\end{equation}
defined in a neighborhood of $\overline{A_j}\times\mathbb{S}^1$, where $\mathbb{S}^1:=\mathbb{R}/T\mathbb{Z}$ for some $T>0$. By defining $\mathcal{H}_i\colon\mathbb{R}^n\times\mathbb{S}^1\to\mathbb{R}$ as $\mathcal{H}_i(x,t)=h_i(x)$, we observe that \eqref{12} describes a piecewise smooth vector field $\mathcal{Z}_0:=(\mathcal{X}_0^1,\dots,\mathcal{X}_0^k)$ on $\mathbb{R}^n\times\mathbb{S}^1$, with $\Omega=\cup_{i=1}^N\Omega_i$ as its switching set, where $\Omega_i:=\mathcal{H}_i^{-1}(\{0\})=\Sigma_i\times\mathbb{S}^1$. Moreover, if $x\in\Sigma_i$ is a crossing point of $Z_0$ for the components $X_0^a$ and $X_0^b$, then $(x,t)\in\Omega_i$ is a crossing point of $\mathcal{Z}_0$ for $\mathcal{X}_0^a$ and $\mathcal{X}_0^b$, for every $t\in\mathbb{S}^1$. Furthermore, if we let $\Phi_0(\tau;x,t)$ denote the solution of $\mathcal{Z}_0$ with initial condition $\Phi_0(0;x,t)=(x,t)$, then we observe that it can be expressed by 
	\[\Phi_0(\tau;x,t)=(\varphi_0(\tau;x),\tau+t),\]
where we recall that $\varphi_0$ is the solution of $Z_0$. Observe that the hyperbolic saddle point $p_0\in\mathbb{R}^n\setminus\Sigma$ of $Z_0$ now corresponds to a hyperbolic periodic orbit 
\begin{equation}\label{18}
	\gamma_0:=\{(x,t)\in\mathbb{R}^n\times\mathbb{S}^1\colon x=p_0\}\subset (\mathbb{R}^n\times\mathbb{S}^1)\setminus\Omega,
\end{equation}
of saddle type of $\mathcal{Z}_0$ (i.e. $W^{s,u}(\gamma_0)\setminus\gamma_0\neq\emptyset)$. Besides that, the intersection between its stable and unstable sets $W^{s}(\gamma_0)$ and $W^{u}(\gamma_0)$ contains the cylindrical surface $\Gamma\times\mathbb{S}^1$, which intersects $\Omega$ in curves consisting only of crossing points (see Figure~\ref{Fig9}$(a)$).
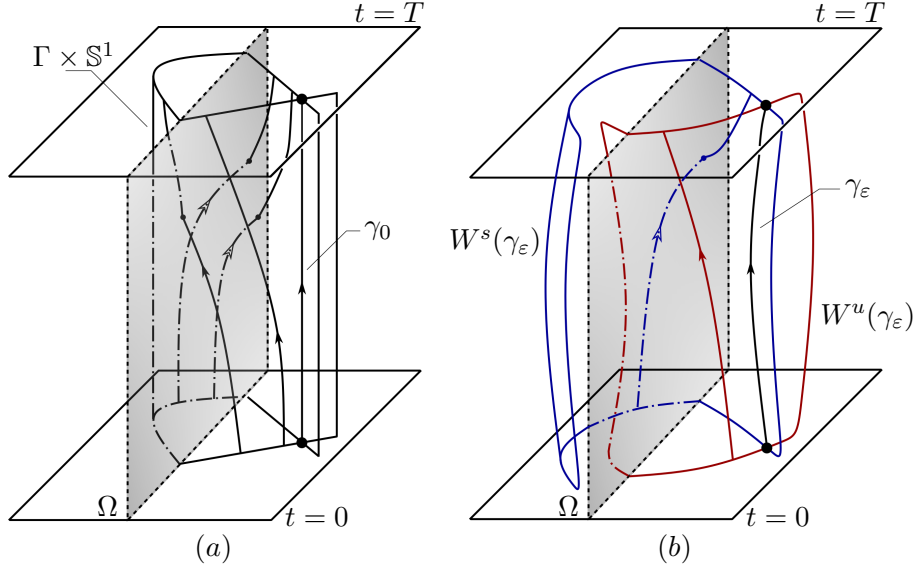
\begin{figure}[ht]
	\begin{center}
		\begin{minipage}{6cm}
			\begin{center} 
				\begin{overpic}[width=5.5cm]{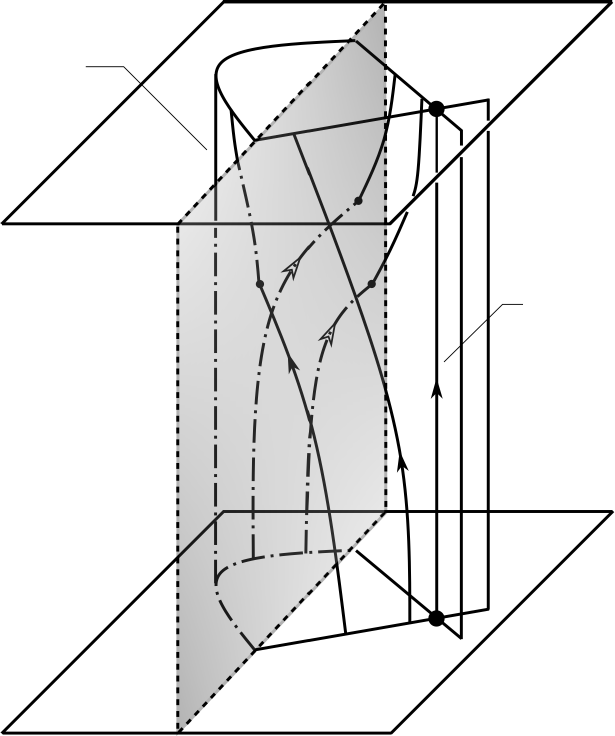} 
					\put(5,92){$\Gamma\times\mathbb{S}^1$}
					\put(72,58){$\gamma_0$}
					\put(18,1){$\Omega$}
					\put(56,-1){$t=0$}
					\put(70,101){$t=T$}
				\end{overpic}
				
				$(a)$
			\end{center}
		\end{minipage}
		\begin{minipage}{6cm}
			\begin{center} 
				\begin{overpic}[width=5.5cm]{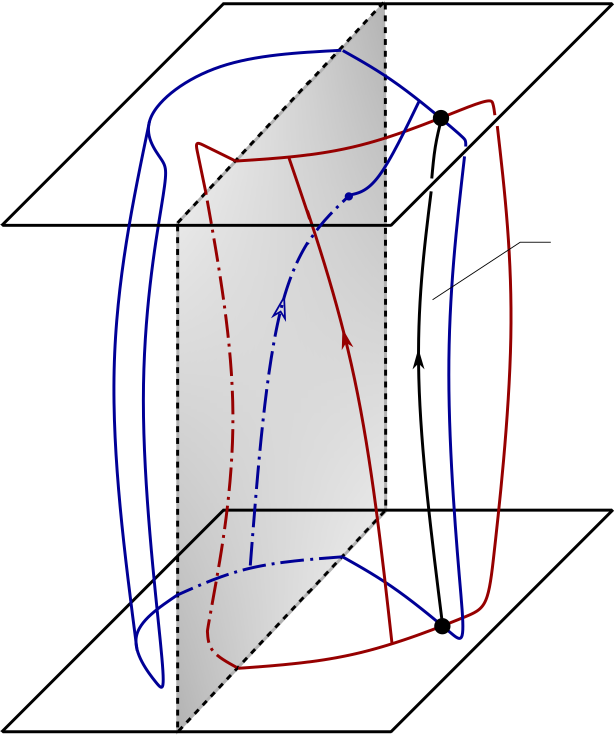} 
					\put(76,66.5){$\gamma_\varepsilon$}
					\put(18,1){$\Omega$}
					\put(56,-1){$t=0$}
					\put(70,101){$t=T$}
					\put(71,40){$W^u(\gamma_\varepsilon)$}
					\put(-4,55){$W^s(\gamma_\varepsilon)$}
				\end{overpic}
				
				$(b)$
			\end{center}
		\end{minipage}
	\end{center}
	\caption{$(a)$ Illustrates the cylindrical surface $\Gamma\times\mathbb{S}^1$ and $(b)$ represents the splitting of $W^{s}(\gamma_0)$ and $W^{u}(\gamma_0)$ when $\mathcal{Z}_0$ is perturbed. The stable (resp. unstable) set of $\gamma_\varepsilon$ is represented in blue (resp. red). Colors available in the online version.}\label{Fig9}
\end{figure}

Given an open set $\mathcal{U}_0\subset\mathbb{R}^n\times\mathbb{S}^1$ and $t_0\in\mathbb{S}^1$, we define $U^{t_0}_0:=\mathcal{U}_0\cap\{t=t_0\}$. Since $\Gamma\times\mathbb{S}^1$ intersects $\Omega$ only in crossing points, it follows that there is a neighborhood $\mathcal{U}_0$ of $\Gamma\times\mathbb{S}^1$ such that for each $(x,t)\in\mathcal{U}_0$, the arc of orbit
\begin{equation}\label{19}
	\{\Phi_0(\tau;x,t)\in\mathbb{R}^{n}\times\mathbb{S}^1\colon\tau\in[-T,T]\},
\end{equation}
is well defined and intersect $\Omega$ only in crossing points (in this claim we are also using Proposition~\ref{P1}, a postponed technical result proved in Section~\ref{Sec3}, which states that such an arc of orbit intersects $\Omega$ at most in finitely many points).

This enables us to construct the time-$T$-map (Poincaré map) associated with $\mathcal{Z}_0$, defined on $U^{t}_0$, which essentially tracks where a solution starting in $U^{t}_0$ will be at time $\tau=T$. Therefore, for each $t\in\mathbb{S}^1$, we define the time-$T$-map associated with $\mathcal{Z}_0$ by
\begin{equation}\label{poincare}
	\begin{array}{rccc}
		P^{t}_0: &U^{t}_0&\longrightarrow& V^{t}_0\\
		&(x,t)&\longmapsto &\Phi_0(T;x,t),
	\end{array} 
\end{equation}
with $V^{t}_0:=P^{t}_0(U_0^{t})\subset\mathbb{R}^n\times\{t\}$. We shall see in Section~\ref{Sec3} that $P^{t}_0$ is a homeomorphism. Moreover, it is also a diffeomorphism except possibly at
\begin{equation}\label{20}
	U^{t}_0\cap\left(\Omega^{t}\cup (P^{t}_0)^{-1}(\Omega^{t})\right),
\end{equation}
where $\Omega^{t_0}:=\Omega\cap\{t=t_0\}$. For simplicity, we shall say that $P^{t}_0$ is a \emph{piecewise diffeomorphism} and identify $(x,t)\approx x$ for every $(x,t)\in U^{t}_0$.

\begin{remark}\label{R1}
	We observe that, unlike in the smooth case, given $x\in U^{t}_0$, we cannot apply $P^{t}_0$ infinitely many times on $x$ because the orbit of $x$ may not remain close to $\Gamma\times\mathbb{S}^1$ for any $\tau>0$. In particular, if such an orbit leaves the domain $U^{t}_0$ after some time, we cannot guarantee that it will keep intersecting $\Omega$ only in crossing points. Consequently, we may lose both the uniqueness of solutions and the well-definition of the Poincaré map.

	Since the orbits with initial condition $x\in\Gamma$ are well defined for every $\tau\in\mathbb{R}$ and remains inside $\mathcal{U}_0$ indefinitely, then, for any $n\in\mathbb{N}$, we can restrict $\mathcal{U}_0$, if necessary, so that for every $(x,t)\in\mathcal{U}_0$, the solution $\Phi_0(\tau;x,t)$ is well defined and remains inside $\mathcal{U}_0$ for all $\tau\in[-nT,nT]$. In particular, we can apply the Poincaré map at least $n$ times for every point in its domain. 
\end{remark}

Let $\varepsilon>0$ be sufficiently small, and for each $1\leq j\leq k$, consider the perturbation of $X_0^j(x)$ given by $X_0^j(x)+\varepsilon X_1^j(x,t)$, with $X_1^j$ smooth in $(x,t)$ and $T$-periodic in $t$. This defines a small perturbation of $Z_0$ denoted by
\begin{equation}\label{28}
	Z_{\varepsilon}(x,t):=\bigl(X_0^1(x)+\varepsilon X_1^1(x,t),\dots,X_0^k(x)+\varepsilon X_1^k(x,t)\bigr),
\end{equation}
with solution $\varphi_\varepsilon(\tau;x,t)$ and initial condition $\varphi_\varepsilon(0;x,t)=x$.

Now, proceeding to the extended vector field by taking $\dot{t}=1$, we notice that $\mathcal{Z}_\varepsilon:=(\mathcal{X}_\varepsilon^1,\dots,\mathcal{X}_\varepsilon^k)$, where $\mathcal{X}_\varepsilon^j$ is given by the product flow
\begin{equation}\label{21}
	\dot x=X_0^j(x)+\varepsilon X_1^j(x,t), \quad \dot t=1,
\end{equation}
is a small perturbation of $\mathcal{Z}_0$. Thus, the solution of $\mathcal{Z}_\varepsilon $ with initial condition $\Phi_\varepsilon(0;x,t)=(x,t) $ is given by 
\begin{equation}\label{23}
	\Phi_\varepsilon(\tau;x,t)=(\varphi_{\varepsilon}(\tau;x,t),\tau+t).
\end{equation}
Similarly to~\eqref{19}, we remark that for each $\varepsilon>0$ small enough, there is a neighborhood $\mathcal{U}_\varepsilon$ of $\Gamma\times\mathbb{S}^1$ such that if $(x,t)\in\mathcal{U}_\varepsilon$, then $\Phi_\varepsilon(\tau;x,t)$ is well defined for every $\tau\in[-T,T]$ and intersects $\Omega$ only in crossing points. This implies that, for each $\varepsilon>0$, the Poincaré map $P_\varepsilon^{t}\colon U^{t}_\varepsilon\to V^{t}_\varepsilon:=P_\varepsilon^{t}(U^{t}_\varepsilon)$ associated with $\mathcal{Z}_\varepsilon$ and given by
\begin{equation}\label{Peps}
		P_\varepsilon^{t}(x)=\Phi_\varepsilon(T;x,t)
\end{equation}
is well defined for every $t\in\mathbb{S}^1$. We observe that for each $|\varepsilon|\geqslant0$ small enough, the domain $U^{t_0}_\varepsilon$ of the Poincaré map $P_\varepsilon^{t_0}$ may change. But it is always defined in a neighborhood of $\Gamma\times\{t=t_0\}$. The precise definition of such a domain is postponed to Section~\ref{Sec3.2}.

Since $\gamma_0$ does not intersect $\Omega$, it follows from the application of the Implicit Function Theorem for $P_\varepsilon^{t}$ (see~\cite[Lemma~$4.5.1$]{GukHol1983}) that $\mathcal{Z}_\varepsilon$ has a unique hyperbolic $T$-periodic orbit of saddle type, $\gamma_\varepsilon$, close to $\gamma_0$. In particular, for sufficiently small $\varepsilon>0$, the periodic orbit $\gamma_\varepsilon$ also does not intersect $\Omega$.

Let $p_\varepsilon^{t_0}:=\gamma_\varepsilon\cap\{t=t_0\}$, which is a fixed point of $P_\varepsilon^{t_0}$. From~\cite[Lemmas~$4.5.1$ and~$4.5.2$]{GukHol1983}, $p_\varepsilon^{t_0}$ is a hyperbolic saddle point of $P_\varepsilon^{t_0}$ for all $t_0\in\mathbb{S}^1$, and the local stable and unstable manifolds $W_{loc}^{s,u}(\gamma_\varepsilon)$, associated with $\mathcal{Z}_\varepsilon$, are also well defined and close to the unperturbed stable and unstable manifolds $W_{loc}^{s,u}(\gamma_0)$ (see Figure~\ref{Fig9}$(b)$). Moreover, we shall see in Section~\ref{Sec3} that 
	\[W^{s,u}(p_\varepsilon^{t_0})=W^{s,u}(\gamma_\varepsilon)\cap\{t=t_0\},\]
and $W^{s,u}(p_\varepsilon^{t_0})\cap U^{t_0}_\varepsilon$ (resp. $W^{s,u}(\gamma_\varepsilon)\cap\mathcal{U}_\varepsilon$) are topological manifolds that fails to be smooth at most in a countable number of points (resp. curves). 

It is worth mentioning that, in the classical theory, the distance between the stable and unstable manifolds can be computed by means of a Melnikov procedure, see~\cite[Section~$4.5$]{GukHol1983}. In the non-smooth context, we have, for instance, the works of Shi et al~\cite{Kupper} and Granados et al~\cite{Tere} for the homoclinic and heteroclinic cases, respectively.

In what follows, we present our main result concerning the homoclinic entanglement under transversal conditions for piecewise smooth maps arising from Filippov vector fields. The proof is postponed to Section~\ref{Sec4}. Since all perturbative discussions so far have been conducted for sufficiently small $\varepsilon>0$, we assume the existence of $\varepsilon^*>0$ such that all objects and conditions depending on $\varepsilon$ remain valid for $0<\varepsilon<\varepsilon^*$. 

\begin{main}[$\lambda$-lemma for homoclinic tangles of PSVF]\label{Main1}
	Let $P_\varepsilon^{t}$ denote the time-$T$-map defined in~\eqref{Peps}, and let $p_{\varepsilon}^{t}$ be its fixed point. Suppose that for some $0<\varepsilon_0<\varepsilon^*$ and $t_0\in\mathbb{S}^1$ there exists a $u$-disk $\Delta\subset W^u(p_{\varepsilon_0}^{t_0})\cap U^{t_0}_{\varepsilon_0}$ intersecting $W^s(p_{\varepsilon_0}^{t_0})\cap U^{t_0}_{\varepsilon_0}$ transversely at some point $q_{\varepsilon_0}^{t_0}\in U^{t_0}_{\varepsilon_0}\setminus\Omega^{t_0}$. Then, for every $u$-disk $D\subset W^u(p_{\varepsilon_0}^{t_0})\cap U^{t_0}_{\varepsilon_0}$ the set
		\[\bigcup_{n=0}^{\infty}(P_{\varepsilon_0}^{t_0})^n(\Delta)\]
	contains $u$-disks arbitrarily close in the $\mathcal{C}^1$-topology (resp. in the $\mathcal{C}^0$-topology) to $D$ when $D\cap\Omega^{t_0}=\emptyset$ (resp. when $D\cap\Omega^{t_0}\neq\emptyset$).
\end{main}

Since $\Delta\subset W^u(p_{\varepsilon_0}^{t_0})$, it follows from Theorem~\ref{Main1} that $W^u(p_{\varepsilon_0}^{t_0})$ and $W^s(p_{\varepsilon_0}^{t_0})$ become entangled in relation to the $\mathcal{C}^1$-topology at all points except where they intersect $\Omega^{t_0}$. Moreover, even at these intersection points, the entanglement persists in the $\mathcal{C}^0$-topology (see Figure~\ref{Fig12}).
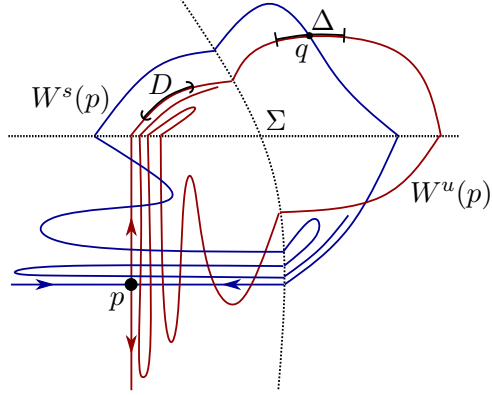
\begin{figure}[ht]
	\begin{center}
		\begin{overpic}[width=6cm]{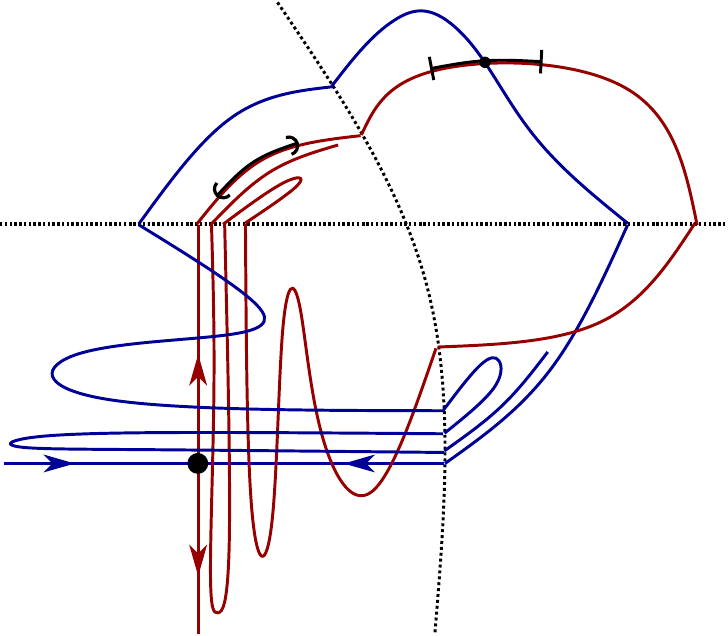} 
			\put(22.5,19){$p$}
			\put(63.5,74){$q$}
			\put(5,63){$W^s(p)$}
			\put(89,42){$W^u(p)$}
			\put(57,58){$\Sigma$}
			\put(67,80){$\Delta$}
			\put(31,66){$D$}
		\end{overpic}
	\end{center}
	\caption{Illustration of Theorem~\ref{Main1}. For simplicity we dropped the symbols $\varepsilon_0$ and $t_0$ in the notation made the identification $\Omega^{t_0}\approx\Sigma$. The stable (resp. unstable) set of $p$ is represented in blue (resp. red). Colors available in the online version.}\label{Fig12}
\end{figure} 

It should be emphasized that the assumptions $D\subset U^{t_0}_{\varepsilon_0}$ and $q_{\varepsilon_0}^{t_0}\in U^{t_0}_{\varepsilon_0}$ are necessary, because we may not have control over $W^{u,s}(p_{\varepsilon_0}^{t_0})$ far away from the homoclinic connection $\Gamma$. More precisely, $U^{t_0}_{\varepsilon_0}$ is the region in which we know that the orbits of $\mathcal{Z}_\varepsilon$ intersects $\Omega$ only in crossing points. Therefore, the uniqueness of solution is preserved and the Poincaré map remain well-defined. Without the assumption $D\subset U^{t_0}_{\varepsilon_0}$, sliding motion may occur, potentially leading to the loss of the well-definition of the Poincaré map.

Moreover, as we shall see in Section~\ref{Sec3}, for each $\varepsilon\in(0,\varepsilon^*)$, any two Poincaré maps $P_\varepsilon^{t_1}$ and $P_\varepsilon^{t_2}$ are topologically conjugated with the conjugation $\Phi_\varepsilon^{t_2-t_1}(\cdot):=\Phi_\varepsilon(t_2-t_1;\cdot)$ being also a piecewise smooth diffeomorphism. Hence we obtain the following result, whose proof is deferred to Section~\ref{Sec4}.

\begin{proposition}\label{MainCoro}
	If $W^u(p_{\varepsilon_0}^{t_0})$ and $W^s(p_{\varepsilon_0}^{t_0})$ intersects transversely at some point $q_{\varepsilon_0}^{t_0}\in U^{t_0}_{\varepsilon_0}\setminus\Omega^{t_0}$ for some $t_0\in\mathbb{S}^1$, then for all $t\in\mathbb{S}^1$, except possibly for finitely many, the sets $W^u(p_{\varepsilon_0}^{t})$ and $W^s(p_{\varepsilon_0}^{t})$ intersect transversely at some point $q_{\varepsilon_0}^{t}\in U^{t}_{\varepsilon_0}\setminus\Omega^{t}$.
\end{proposition}

In other words, given $\varepsilon_0\in(0,\varepsilon^*)$, it follows from Proposition~\ref{MainCoro} that if some Poincaré map $P_{\varepsilon_0}^{t_0}$ satisfies the hypothesis of Theorem~\ref{Main1}, then all but finitely many Poincaré maps $P_{\varepsilon_0}^{t}$ also satisfy it. Geometrically this means that if $W^s(p_{\varepsilon_0}^{t})$ and $W^u(p_{\varepsilon_0}^{t})$ are entangled for some $t\in\mathbb{S}^1$, then $W^s(\gamma_{\varepsilon_0})$ and $W^u(\gamma_{\varepsilon_0})$ are also entangled (see Figure~\ref{Fig13}). 
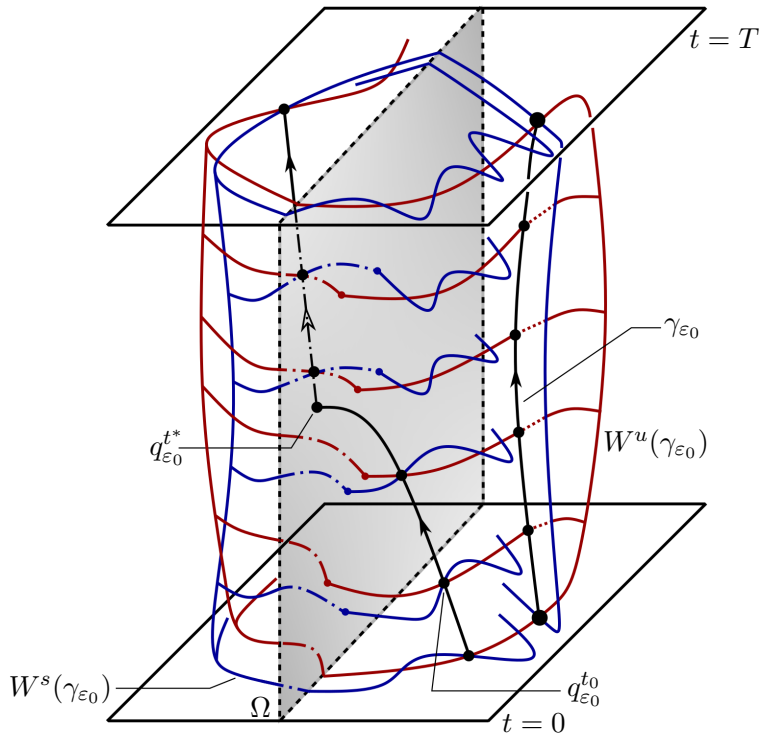
\begin{figure}[ht]
	\begin{center}
		\begin{overpic}[width=8cm]{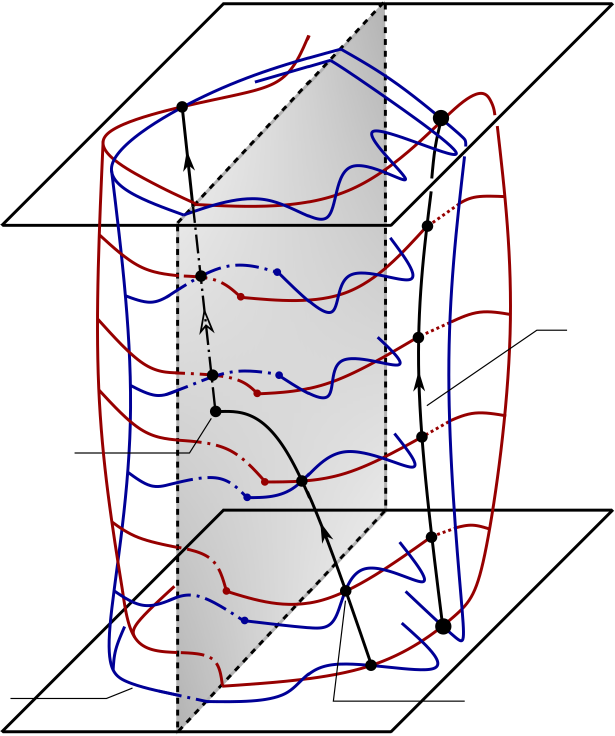} 
			\put(77.5,54.75){$\gamma_{\varepsilon_0}$}
			\put(20,1){$\Omega$}
			\put(55,-1){$t=0$}
			\put(81,94){$t=T$}
			\put(69,38){$W^u(\gamma_{\varepsilon_0})$}
			\put(-13.5,4){$W^s(\gamma_{\varepsilon_0})$}
			\put(64,4){$q_{\varepsilon_0}^{t_0}$}
			\put(6,37.5){$q_{\varepsilon_0}^{t^*}$}
		\end{overpic}
	\end{center}
	\caption{Illustration of the entanglement between $W^s(\gamma_{\varepsilon_0})$ (in blue) and $W^s(\gamma_{\varepsilon_0})$ (in red). Colors available in the online version.}\label{Fig13}
\end{figure}

For an illustration of the smooth version of this entanglement, we refer to the beautiful Figure~$3.40$ of~\cite[p.~$175$]{Arrow}.

Additionally, from the proof of Proposition~\ref{MainCoro}, it becomes clear that if the intersection fails to be transversal for some $t^*\in\mathbb{S}^1$, then $q_{\varepsilon_0}^{t^*}\in\Omega^{t^*}$ (see Figure~\ref{Fig13}).

\section{Preliminaries results}\label{Sec3}

\subsection{Piecewise smooth vector fields}\label{Sec3.1}

In this sub-section, we establish some technical results concerning PSVF. For sake of generality, we suppose only that the vector fields treated are autonomous, such as $\mathcal{Z}_0$ and $\mathcal{Z}_\varepsilon$ in Section~\ref{Sec2} (see~\eqref{12} and \eqref{21}), but not necessarily given by a product flow, such as the vector fields that shall be treated at Section~\ref{Sec5}. 

Let $Z$ be an autonomous piecewise smooth vector field on $\mathbb{R}^n$, $n\geqslant2$, with components $(X_1,\dots,X_k)$ and switching set $\Sigma=\cup_{i=1}^{N}\Sigma_i$ with $\Sigma_i=h^{-1}_i(\{0\})$, with $h_i\colon\mathbb{R}^n\to\mathbb{R}$ being a smooth function.

As mentioned in Section~\ref{Sec2}, in this paper we are not dealing  with non-crossing points in $\Sigma$. Therefore if $x\in\mathbb{R}^n$ is such that its trajectory intersects $\Sigma$, then we shall assume that there is a compact interval of time such that it does only in crossing points. This enables us to define a global solution $\varphi(t;x)$ of $Z$, with initial condition $\varphi(0;x)=x$, given by the concatenation (in both forward and backward time) of the pieces of solutions of its components $(X_1,\dots,X_k)$. The orbit of $Z$ through $x$ is denoted by $\gamma_x:=\{\varphi(t;x):t\in I_x\}$, where $I_x\subset\mathbb{R}$ is a compact interval containing $0$ in its interior and such that $\varphi(\cdot;x)$ exists for every $t\in I_x$ and whenever it crosses $\Sigma$, it does in crossing points. This implies that for each $T\geqslant0$ (resp. $T\leqslant0$), if $x\in\mathbb{R}^n$ is such that $[0,T]\subset\operatorname{Int}(I_x)$ (resp. $[T,0]\subset\operatorname{Int}(I_x)$), where $\operatorname{Int}(I)$ denotes the topological interior of $I$, then there exists a neighborhood $A^T\subset\mathbb{R}^n$ of $x$ such that the time-$T$-map $\Phi^T\colon A^T\to\mathbb{R}^n$, given by $\Phi^T(r)=\varphi(T;r)$, is well defined. In particular, if $T=0$, then $\Phi^0$ is the identity map.

In order to simplify the writing, let $(Z,\Sigma,\varphi)$ denote an autonomous piecewise smooth vector field $Z$ with switching set $\Sigma$ and global solution $\varphi$. Then, given $T\in\mathbb{R}$, we denote by $\Phi^T$ the time-$T$-map $\Phi^T(x):=\varphi(T;x)$, whenever it is well defined. 

In the following proposition we verify that crossing solutions intersect the switching finitely many times in compact time intervals.

\begin{proposition}\label{P1}
	Consider $(Z,\Sigma,\varphi)$, $x\in\mathbb{R}^n\setminus\Sigma$ and $T>0$ such that $\varphi(T,x)\not\in\Sigma$. If $\{\varphi(t;x):t\in[0,T]\}$ intersects $\Sigma$ only in crossing points, then such an intersection is finite.
\end{proposition}

\begin{proof}
	Suppose by contradiction that $\{\varphi(t;x):t\in[0,T]\}\cap\Sigma$ is not finite. Then there exists an increasing sequence $(t_k)_{k\in\mathbb{N}}\subset[0,T]$ such that $\varphi(t_k;x)\in\Sigma$ for every $k\in\mathbb{N}$. Let 
		\[t^*:=\lim\limits_{k\to\infty}t_k=\sup\{t_k\colon k\in\mathbb{N}\}<T,\]
	and observe that $t_k-t_{k-1}\to0$ as $k\to\infty$. Let $x_k=\varphi(t_k;x)$, $k\in\mathbb{N}$, and $x^*=\varphi(t^*;x)$. Since $\Sigma$ is a closed set, it follows that $x^*\in\Sigma$ and thus by hypothesis it is also a crossing point.
	
	Observe that we can take a local system of coordinates in a neighborhood $W\subset\mathbb{R}^n$ of $x^*$ such that in $W$ we have $\Sigma=h^{-1}(\{0\})$ and $Z=(X^+,X^-)$, with $X^\pm$ defined in a neighborhood of the region
		\[\Sigma^\pm=\{x\in\mathbb{R}^n\colon \pm h(x)\geqslant0\},\]
	for some $\mathcal{C}^\infty$-function $h\colon W\to\mathbb{R}$. Suppose for definiteness that the local trajectory of $Z$ through $x^*$ enters $\Sigma^-$. Restricting $W$, if necessary, we can assume that the local trajectory of $Z$ through every $r\in W\cap\Sigma$ also enters $\Sigma^-$ (see Figure~\ref{Fig5}).
	\begin{figure}[ht]
		\begin{center}
			\begin{overpic}[width=8cm]{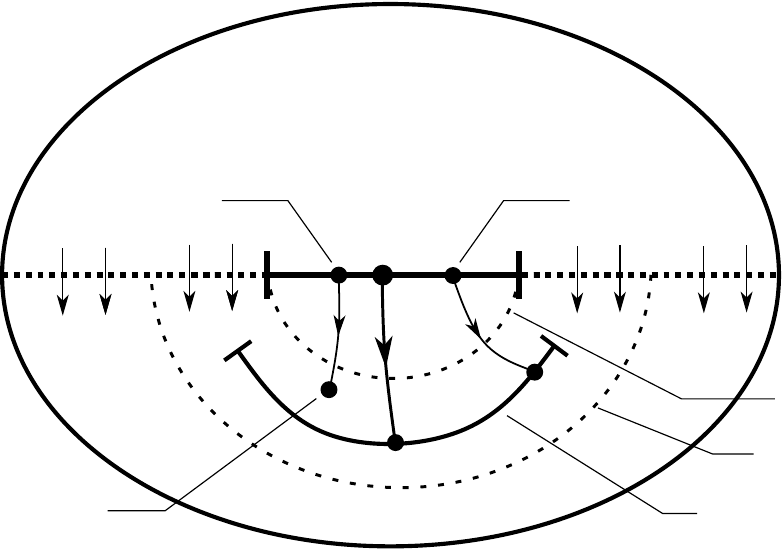} 
				\put(85,61){$W$}
				\put(2,36){$\Sigma$}
				\put(11,14){$\Sigma^-$}
				\put(48,64){$\Sigma^+$}
				\put(48,38){$x^*$}
				\put(23,44){$x_k$}
				\put(4,4){$x_{k+1}$}
				\put(74,43.5){$r$}
				\put(99.5,18){$A$}
				\put(97.5,11){$U$}
				\put(90,4){$\varphi^-(\eta;A\cap\Sigma)$}
			\end{overpic}
		\end{center}
		\caption{Illustration of the proof of Proposition~\ref{P1}.}\label{Fig5}
	\end{figure}
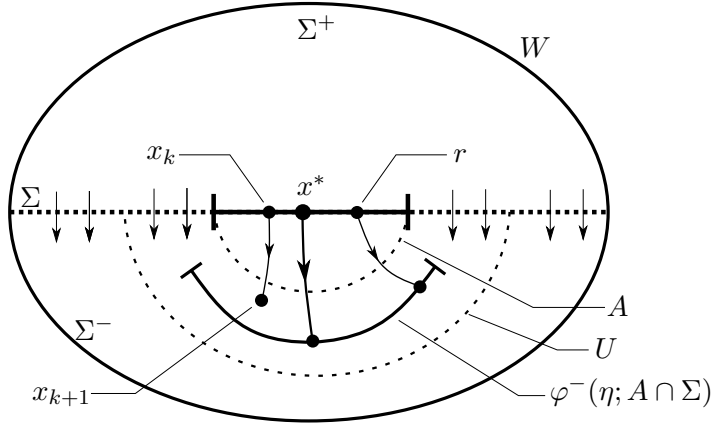
	
	Let $\varphi^-(\tau;r)$ be the solution of $X^-$ with initial condition $\varphi^-(0;r)=r$. Given a neighborhood $U\subset W$ of $x^*$, it follows from the continuous dependence on the initial conditions (see~\cite[Theorem $8$ and $9$]{Andronov}) that there exist a neighborhood $A\subset U$ of $x^*$ and $\eta>0$ such that if $(\tau,r)\in[0,\eta]\times( A\cap\Sigma^-)$, then $\varphi^-(\tau,r)\in U\cap\Sigma^-$.
	
	Restricting $A$ if necessary, we have from the Flow Box Theorem (see~\cite[Theorem~$1.12$]{DumLliArt2006}) that if $(\tau,r)\in(0,\eta]\times (A\cap\Sigma)$, then $\varphi^-(\tau,r)\in (U\cap\Sigma^-)\setminus\Sigma$ (see Figure~\ref{Fig5}).
	
	Since $x_k\to x^*$ and $t_k-t_{k-1}\to0^+$, it follows that for $k\in\mathbb{N}$ big enough we have $x_k\in A\cap\Sigma$ and $t_{k+1}-t_k<\eta$, which implies that $x_{k+1}\in(U\cap\Sigma^-)\setminus\Sigma$, contradicting the fact that $x_{k+1}\in\Sigma$ is a crossing point.
\end{proof}

\begin{proposition}\label{P2}
	Consider $(Z,\Sigma,\varphi)$ and $x\in\mathbb{R}^n\setminus\Sigma$. Let $T>0$ be such that $\{\varphi(t;x):t\in[0,T]\}$ intersects $\Sigma$ only in crossing points and $\varphi(T;x)\not\in\Sigma$. Then $\Phi^T$ is a $\mathcal{C}^\infty$-diffeomorphism in a neighborhood of $x$.
\end{proposition}

\begin{proof}
	Since $x\not\in\Sigma$ and $\varphi(T;x)\not\in\Sigma$, we can extend the orbit of $Z$ through $x$ to a neighborhood $I_x\subset\mathbb{R}$ of $[0,T]$ such that $\{\varphi(t;x):t\in I_x\}$ intersects $\Sigma$ only in crossing points. In particular, $\Phi^T$ is well defined in a neighborhood of $x$. 
	
	From Proposition~\ref{P1} there are at most finitely many $t_1,\dots,t_k\in(0,T)$ such that $\varphi(t_i;q)\in\Sigma$. We assume for the moment that $k=1$. In this case, we can also assume that there exists a $\mathcal{C}^\infty$-function $h\colon\mathbb{R}^n\to\mathbb{R}$ such that $\Sigma=h^{-1}(\{0\})$ and $Z=(X^+,X^-)$, with $X^\pm$ defined in a neighborhood of the region
		\[\Sigma^\pm=\{x\in\mathbb{R}^n\colon \pm h(x)\geqslant0\},\]
	with the corresponding solution denoted by $\varphi^\pm(\tau;r)$. Suppose for definiteness that $x\in\Sigma^+$ and $\varphi(T;x)\in\Sigma^-$ (see Figure~\ref{Fig1}). 
	\begin{figure}[ht]
		\begin{center}
			\begin{overpic}[width=5cm]{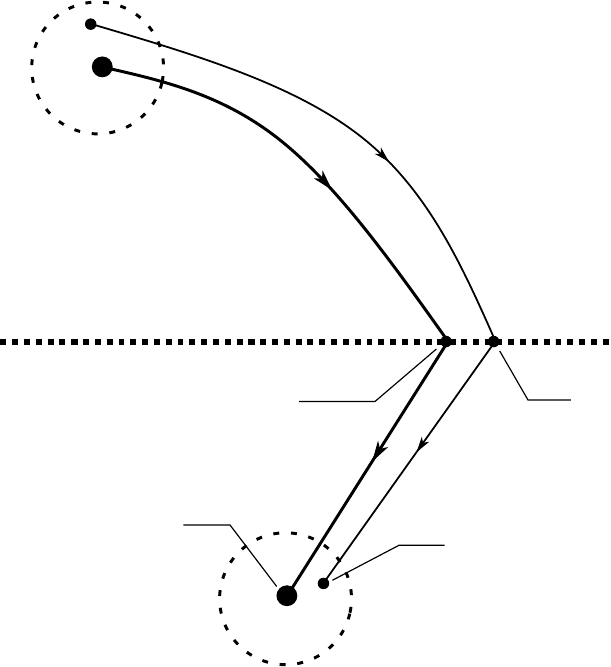} 
				\put(-2,51){$\Sigma$}
				\put(20,77){$A$}
				\put(11,85){$x$}
				\put(9,93){$r$}
				\put(68,16){$\Phi^T(r)$}
				\put(8,19.5){$\Phi^{T}(x)$}
				\put(87,39){$\varphi(\tau(r);r)$}
				\put(21,38){$\varphi(t_1;x)$}
				\put(52,0){$\Phi^T(A)$}
			\end{overpic}
		\end{center}
		\caption{Illustration of the proof of Proposition~\ref{P2}.}\label{Fig1}
	\end{figure}	
	 Since $X^+$ is defined in a neighborhood of $\Sigma^+$, there exists $\eta>0$ sufficiently small such that $\varphi^+(\tau;r)$ is well defined for $\tau\in I:=(-\eta,t_1+\eta)$. Hence, it follows from the continuous dependence on the initial conditions (see~\cite[Theorem $8$ and $9$]{Andronov}) that there is a neighborhood $A\subset\mathbb{R}^n\setminus\Sigma$ of $x$ such that $\varphi^+(\tau;r)$ is well defined for every $(\tau;r)\in I\times A$. Let $F\colon I\times A\to\mathbb{R}$ be given by
		\[F(t,r)=h(\varphi^+(t;r)).\]
	Notice that $F$ is well defined and of class $\mathcal{C}^\infty$ in $I\times A$. Moreover, 
	\begin{equation}\label{1}
		\frac{\partial F}{\partial t}(t,r)=\left<\nabla h(\varphi^+(t;r)),X^+(\varphi^+(t;r))\right>,
	\end{equation}
	where we have used that $\partial\varphi^+/\partial t=X^+\circ\varphi^+$. Thus, Equation \eqref{1} in addition with conditions $(ii)$ and $(iii)$ from Definition~\ref{Def1} lead us to 
		\[F(t_1,x)=0\quad \mbox{and} \quad \frac{\partial F}{\partial t}(t_1,x)\neq0.\]
	Therefore, it follows from the Implicit Function Theorem that there is a $\mathcal{C}^\infty$-function $\tau\colon A\to\mathbb{R}$, with $\tau(q)=t_1$, such that $F(\tau(r),r)\equiv0$. Observe that $\tau(r)$ is exactly the time necessary for a point $r\in A$ to reach $\Sigma$. Restricting $A$ if necessary, we can assume that there is $\tau_0>0$ such that $\tau(r)\leqslant\tau_0<T$ for every $r\in A$. This implies that, for $r\in A$, the time-$T$-map $\Phi^T$ can be written as
		\[\Phi^T(r)=\varphi^-(T-\tau(r);\varphi^+(\tau(r);r)).\]
	In particular, $\Phi^T$ is $\mathcal{C}^\infty$ when restricted to $A$. By reversing the time, we observe that the endpoint $\Phi^{T}(x)$ satisfies the same hypothesis as $x$ and thus we can also prove that $\Phi^{-T}$ is $\mathcal{C}^\infty$ in a neighborhood of $\Phi^{T}(x)$. This concludes the proof, since $\Phi^{T}$ and $\Phi^{-T}$ are inverses one of each other.
	
	In order to prove the general case (i.e., $k>1$), let us set $T_0=0$ and $T_k=T$, and for each $i\in\{1,\dots,k-1\}$, choose $T_i\in(t_i,t_{i+1})$. Also, let $x_i=\varphi(T_i;x)$ for $i\in\{0,\dots,k-1\}$. Analogously to the previous reasoning, the map $\Phi^{T_{i+1}-T_i}$ is a $\mathcal{C}^\infty$-difeomorphism in a neighborhood of $x_i$, ${i\in\{0,\dots,k-1\}}$. The proof now follows from the fact that $\Phi^T=\Phi^{T_k-T_{k-1}}\circ\dots\circ\Phi^{T_1-T_0}$.
\end{proof}

The following example highlights the importance of assuming that the starting and ending points of an arc of orbit do not belong to the switching set, which is essential to ensure that the time-$T$-map is a local diffeomorphism at such points, as emphasized in the previous proposition.

\begin{example}
Let us consider the planar piecewise smooth vector field $Z=(X^+,X^-)$, where  $X^+(x,y)=(0,1)$ and $X^+(x,y)=(1,1)$ are defined in an open neighborhood of the region
\[
\Sigma^\pm\{(x,y)\in\mathbb{R}^2\colon \pm y\geqslant0\},
\]
meaning that
\[
\Sigma=\{(x,y)\in\mathbb{R}^2\colon y=0\},
\]
represents the switching set of $Z$.

We notice that the solutions of $X^+$ and $X^-$ are given by
	\begin{equation}\label{2}
		\varphi^+(t;x,y)=(x,t+y)\quad \mbox{and} \quad  \varphi^-(t;x,y)=(t+x,t+y),
	\end{equation}
	respectively. From \eqref{2}, it is not hard to see that the solution $\varphi(t;x,y)$ of $Z$ is given by
	\begin{equation}\label{3}
		\varphi(t;x,y)=\left\{\begin{array}{ll}
			(t+x,t+y) & \text{if } y\leqslant0 \text{ and } t\leqslant -y, \vspace{0.2cm} \\
			(x-y,t+y) & \text{if } y\leqslant0 \text{ and } t\geqslant -y, \vspace{0.2cm} \\
			(x,t+y) & \text{if } y>0 \text{ and } t\geqslant -y, \vspace{0.2cm} \\
			(t+x+y,t+y) & \text{if } y>0 \text{ and } t\leqslant -y.
		\end{array}\right.
	\end{equation}
	Given $T>0$, it follows from \eqref{3} that
		\[\lim\limits_{y\to0^+}\frac{\Phi^T(0,y)-\Phi^T(0,0)}{y}=(0,1), \quad \lim\limits_{y\to0^-}\frac{\Phi^T(0,y)-\Phi^T(0,0)}{y}=(-1,1).\]
	In particular, it follows that $\partial\Phi^T/\partial y$ is not well defined at $(0,0)$ and thus $\Phi^T$ is not differentiable at the origin, for every $T>0$.
\end{example}

Although the time-$T$-map of $Z$ may not be everywhere smooth, we can prove that it is continuous.

\begin{proposition}\label{P3}
	Consider $(Z,\Sigma,\varphi)$ and $x\in\mathbb{R}^n$. Let $T>0$ be such that $\{\varphi(t;x):t\in[0,T]\}$ intersects $\Sigma$ only in crossing points. Then $\Phi^T$ is a homeomorphism in a neighborhood of $x$.
\end{proposition}

\begin{proof}
	Observe that even if $x\in\Sigma$ or $\varphi(T;x)\in\Sigma$, we can extend the orbit of $Z$ through $x$ to a neighborhood $I_x\subset\mathbb{R}$ of $[0,T]$ such that $\{\varphi(t;x):t\in I_x\}$ intersects $\Sigma$ only in crossing points. In particular, $\Phi^T$ is well defined in a neighborhood of $x$.
	
	If $\{\varphi(t;x):t\in[0,T]\}$ never intersects $\Sigma$ then there is nothing to prove. On the other hand, suppose that 
		\[\{\varphi(t;x):t\in[0,T]\}\cap\Sigma\neq\emptyset.\] 
	If $x\not\in\Sigma$ and $\varphi(T;x)\not\in\Sigma$, then the result follows from Propositions~\ref{P1} and~\ref{P2}. Hence we need to work the cases in which $x\in\Sigma$ or $\varphi(T;x)\in\Sigma$.
	
	Firstly we suppose that $x\not\in\Sigma$ and $\varphi(T;x)\in\Sigma$. For simplicity, we can also assume that $\varphi(t;x)\not\in\Sigma$ for every $t\in(0,T)$. Similarly to the proof of Proposition~\ref{P1}, there exists a neighborhood $A\subset\mathbb{R}^n\setminus\Sigma$ of $x$ and a $\mathcal{C}^\infty$-function $\tau\colon A\to\mathbb{R}$, with $\tau(x)=T$, such that $\varphi(\tau(r),r)\in\Sigma$ for every $r\in A$. Hence, $\Phi^T|_A$ can be written as
		\[\Phi^T|_A(r)=\left\{\begin{array}{ll}
			\varphi^-(T-\tau(r),\varphi^+(\tau(r),r)) & \text{if } \tau(r)\leqslant T, \vspace{0.2cm} \\
			\varphi^+(T,r) & \text{if } \tau(r)\geqslant T,
		\end{array}\right.\]
	where $\varphi^\pm(\tau;r)$ is the solution of $X^\pm$ (see Figure~\ref{Fig2}$(a)$).
	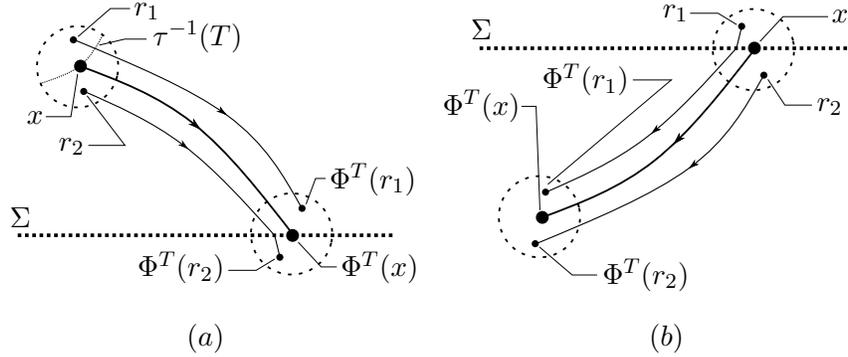
\begin{figure}[ht]
		\begin{center}
			\begin{minipage}{6cm}
				\begin{center} 
					\begin{overpic}[width=5cm]{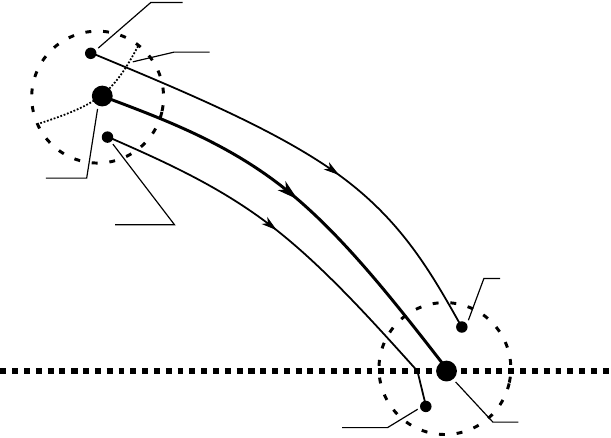} 
						\put(-2,13){$\Sigma$}
						\put(31,69.5){$r_1$}
						\put(36,61.5){$\tau^{-1}(T)$}
						\put(2.5,40.5){$x$}
						\put(11,33.5){$r_2$}
						\put(83,23.5){$\Phi^T(r_1)$}
						\put(86,0){$\Phi^T(x)$}
						\put(32,0){$\Phi^T(r_2)$}
					\end{overpic}
				
					$\;$
					
					$(a)$
				\end{center}
			\end{minipage}
			\begin{minipage}{6cm}
				\begin{center} 
					\begin{overpic}[width=5cm]{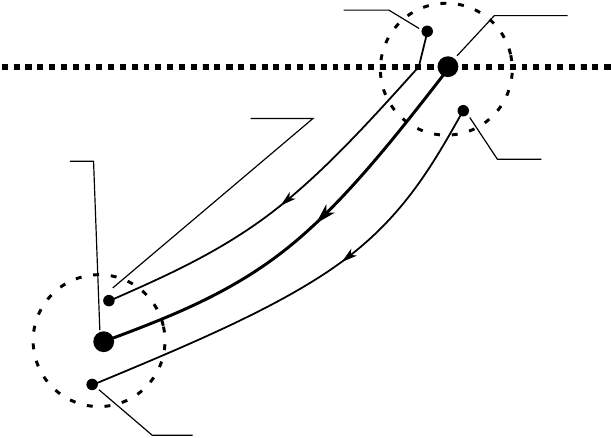} 
						\put(-2,63){$\Sigma$}
						\put(49,69){$r_1$}
						\put(94,67.5){$x$}
						\put(90,44){$r_2$}
						\put(17,50){$\Phi^T(r_1)$}
						\put(33,-2){$\Phi^T(r_2)$}
						\put(-9.5,43){$\Phi^T(x)$}
					\end{overpic}
				
					$\;$
					
					$(b)$
				\end{center}
			\end{minipage}
		\end{center}
		\caption{Illustration of the proof of Proposition~\ref{P3}. In both figures we have $\tau(r_1)>\tau(x)$ and $\tau(r_2)<\tau(x)$.}\label{Fig2}
	\end{figure}
	Therefore it follows from the Pasting Lemma (see~\cite[Theorem $18.3$]{Munkres}) that $\Phi^T$ is continuous in $A$. The case in which $x\in\Sigma$ and $\varphi(T;x)\not\in\Sigma$ follows similarly, with the difference that $\tau(x)=0$, leading to
		\[\Phi^T|_A(r)=\left\{\begin{array}{ll}
			\varphi^-(T-\tau(r),\varphi^+(\tau(r),r)) & \text{if } \tau(r)\geqslant 0, \vspace{0.2cm} \\
			\varphi^-(T,r) & \text{if } \tau(r)\leqslant 0,
		\end{array}\right.\]
	where $A$ is restricted so that $|\tau(r)|\leqslant\tau_0<T$ for every $r\in A$ (see Figure~\ref{Fig2}$(b)$).	The case where $x\in\Sigma$ and $\varphi(T;x)\in\Sigma$ follows similarly. 
	
	Observe that $\Phi^{-T}$ satisfies similar hypothesis and thus can also be proved to be continuous in a neighborhood of $\varphi(T;x)$. The result now follows from the fact that $\Phi^{-T}$ is the inverse of $\Phi^T$.
	
	The general case in which there are finitely many (due to Proposition~\ref{P1}) $t_1,\dots,t_k\in(0,T)$ such that $\varphi(t_k;x)\in\Sigma$ is analogous to what was done in the end of the proof of Proposition~\ref{P2}.
\end{proof}

\subsection{Poincar\'e map}\label{Sec3.2}

We now return to the original context of Section~\ref{Sec2}, concerning the homoclinic connection and its surrounding neighborhood. Given a set $\mathcal{W}\subset\mathbb{R}^n\times\mathbb{S}^1$ and $t_0\in\mathbb{S}^1$, we recall that
	\[W^{t_0}=\{(x,t)\in\mathcal{W}\colon t=t_0\}.\]
Let $\Phi_\varepsilon^T$ be the time $T$-map of $\mathcal{Z}_\varepsilon$, that is, $\Phi_\varepsilon^T(x,t)=\Phi_\varepsilon(T;x,t)$, where $\Phi_\varepsilon$ is the global solution \eqref{23} of $\mathcal{Z}_\varepsilon$. It follows from Propositions~\ref{P1}, \ref{P2} and~\ref{P3} that, for $\varepsilon>0$ small enough, there are neighborhoods $\mathcal{A}_\varepsilon$, $\mathcal{B}_\varepsilon\subset\mathbb{R}^n\times\mathbb{S}^1$ of $\Gamma\times\mathbb{S}^1$ (where we recall that $\Gamma$ is the homoclinic connection of $Z_0$) such that $\Phi_\varepsilon^{T}\colon\mathcal{A}_\varepsilon\to\mathcal{B}_\varepsilon$ is a well defined homeomorhism.

Hence, for each $\varepsilon>0$, we can take a small neighborhood $U_\varepsilon^0\subset A_\varepsilon^0$ of $\Gamma\times\{t=0\}$ such that if we define

\[
	\mathcal{U}_\varepsilon:=\bigcup_{t\in[0,T]}\Phi_\varepsilon^{t}(U_\varepsilon^{0}), \quad \mathcal{V}_\varepsilon:=\bigcup_{t\in[0,T]}\Phi_\varepsilon^{T+t}(U_\varepsilon^{0}),
\]
 then $\mathcal{U}_\varepsilon\subset\mathcal{A}_\varepsilon$ and $\mathcal{V}_\varepsilon\subset\mathcal{B}_\varepsilon$. In particular, for $\varepsilon>0$ small enough, we have
\begin{equation}\label{13}
	U_\varepsilon^{t_2}=\Phi_\varepsilon^{t_2-t_1}(U_\varepsilon^{t_1}), \quad V_\varepsilon^{t_1}=\Phi_\varepsilon^{t_1-t_2}(V_\varepsilon^{t_2}),
\end{equation}
for any two $t_1$, $t_2\in\mathbb{S}^1$.

\begin{proposition}\label{P4}
		Let $P_\varepsilon^t\colon U^t_\varepsilon\to V^t_\varepsilon $ be the time-$T$-map defined in~\eqref{Peps}. For sufficiently small $\varepsilon>0$, the following conditions hold:
	\begin{itemize}
		\item[(a)] $P_\varepsilon^t$ is a homeomorphism for every $t\in\mathbb{S}^1 $;
		\item[(b)] $P_\varepsilon^t$ is a diffeomorphism in $U_\varepsilon^t\setminus\bigl(\Omega^t\cup (P_\varepsilon^t)^{-1}(\Omega^t)\bigr)$.
	\end{itemize}
\end{proposition}

\begin{proof}
	Statement (a) follows directly from Propositions~\ref{P1} and \ref{P3}. In order to prove statement (b), we notice that if  $x\in U^t_\varepsilon\setminus\bigl(\Omega^t\cup (P_\varepsilon^t)^{-1}(\Omega^t)\bigr)$, then $x\notin \bigl(\Omega^t\cup (P_\varepsilon^t)^{-1}(\Omega^t)\bigr)$. Therefore, $x\not\in\Omega^t$ and $P_\varepsilon^t(x)\not\in\Omega^t$, and consequently, taking into account Propositions~\ref{P1} and~\ref{P2}, the map $P_\varepsilon^t$ is a smooth diffeomorphism in a neighborhood of $x$. In particular, we conclude that $P_\varepsilon^t$ is a smooth diffeomorhism in $U_\varepsilon^t\setminus\bigl(\Omega^t\cup (P_\varepsilon^t)^{-1}(\Omega^t)\bigr)$.
\end{proof}

We notice from \eqref{13} that the diagram
	\[
	\begin{tikzcd}
		U^{t_1}_\varepsilon \arrow{r}{P_\varepsilon^{t_1}} \arrow{d}[swap]{\Phi_\varepsilon^{t_2-t_1}} & V^{t_1}_\varepsilon 	\arrow{d}{\Phi_\varepsilon^{t_2-t_1}} \\
		U^{t_2}_\varepsilon \arrow{r}[swap]{P_\varepsilon^{t_2}} & V^{t_2}_\varepsilon
	\end{tikzcd}
	\]
	is commutative. This means that for each $\varepsilon>0$ small, given any two $t_1,t_2\in\mathbb{S}^1$, the Poincaré maps $P_{\varepsilon}^{t_1}$ and $P_{\varepsilon}^{t_2}$ are topologically conjugated. Besides that, from Proposition~\ref{P1} and~\ref{P2} we have (similarly to the proof of Proposition~\ref{P4}) that $\Phi_{\varepsilon}^{t_2-t_1}$ is a diffeomorphism in $U_{\varepsilon}^{t_1}\setminus\bigl(\Omega^{t_1}\cup\Phi_\varepsilon^{t_1-t_2}(\Omega^{t_2})\bigr)$. This leads us to the following proposition.
	
\begin{proposition}\label{P5}
	Let $\varepsilon>0$ small enough. Given any two $t_1$, $t_2\in\mathbb{S}^1$, the maps $P_\varepsilon^{t_1}$ and $P_\varepsilon^{t_2}$ are topologically conjugated, with the conjugation being also a piecewise smooth diffeomorphism.
\end{proposition}

In light of Proposition~\ref{P5}, any two Poincaré maps are conjugated. Therefore, for simplicity, in the remaining results of this section we will consider $P_\varepsilon:=P_\varepsilon^0$ and $p_\varepsilon:=p_\varepsilon^0$. In addition, when referring to the stable and unstable sets $W^{s,u}(\gamma_\varepsilon)\cap\mathcal{U}_\varepsilon$ and $W^{s,u}(p_\varepsilon)\cap U_\varepsilon^0$ in the next propositions, we will only write $W^{s,u}(\gamma_\varepsilon)$ and $W^{s,u}(p_\varepsilon)$ for simplicity. 

\begin{proposition}\label{P6}
	Consider $\varepsilon>0$ small enough. Let $\gamma_\varepsilon$ be the periodic orbit of $\mathcal{Z}_\varepsilon$ and $p_{\varepsilon}$ the associated fixed point of $P_\varepsilon$. Then, the stable and unstable sets of $\gamma_{\varepsilon}$ and $p_{\varepsilon} $ satisfy $W^{s,u}(p_\varepsilon)=W^{s,u}(\gamma_\varepsilon)\cap\{t=0\}$.
\end{proposition}

\begin{proof}
	We recall that $\Phi_\varepsilon(\tau;x,t)$ is the solution of $\mathcal{Z}_\varepsilon$ with initial condition $\Phi_\varepsilon(0;x,t)=(x,t)$. Thus, observe that $\gamma_{\varepsilon}=\{\Phi_\varepsilon(\tau;p_{\varepsilon}):\tau\in\mathbb{R}\}$. Given $(r,0)\in W^s(\gamma_\varepsilon)\cap\{t=0\}$, we have
	\begin{equation}\label{14}
		\lim\limits_{\tau\to+\infty}|\Phi_\varepsilon(\tau;r,0)-\Phi_\varepsilon(\tau;p_{\varepsilon})|=0.
	\end{equation}
	In particular, it follows that
	\begin{equation}\label{15}
		\lim\limits_{n\to+\infty}\Phi_\varepsilon(nT;r,0)=p_\varepsilon,
	\end{equation}
	and thus $(r,0)\in W^s(p_\varepsilon)$. Reciprocally, given $(r,0)\in W^s(p_\varepsilon)$, we have that \eqref{15} holds. Suppose by contradiction that \eqref{14} does not hold. Then, there exist $\eta>0$ and a sequence $(\tau_n)_{n\in\mathbb{N}}\subset\mathbb{R}$, $\tau_n\to+\infty$, such that 
	\begin{equation}\label{16}
		|\Phi_\varepsilon(\tau_n;r,0)-\Phi_\varepsilon(\tau_n;p_{\varepsilon})|>\eta,
	\end{equation}
	for every $n\in\mathbb{N}$. For each $n\in\mathbb{N}$, let $t_n\in\mathbb{S}^1$ and $k_n\in\mathbb{Z}$ be such that $\tau_n=t_n+k_nT$. Since $p_\varepsilon^t$ is a fixed point of $P_\varepsilon^t$, for every $t\in\mathbb{S}^1$, it follows that $\Phi_\varepsilon(\tau_n;p_{\varepsilon})=p_\varepsilon^{t_n}$. Hence, we have from \eqref{16} that
	\begin{equation}\label{17}
		|\Phi_\varepsilon\bigl(t_n;(P_\varepsilon)^{k_n}(r,0)\bigr)-p_\varepsilon^{t_n}|>\eta,
	\end{equation}
	for every $n\in\mathbb{N}$.	Passing to a sub-sequence if necessary, we can assume that there exists $t^*\in\mathbb{S}^1$ such that $t_n\to t^*$. Hence, from relation~\eqref{17}, it follows that
	\begin{equation}\label{22}
		|\Phi_\varepsilon\bigl(t^*;(P_\varepsilon)^{k_n}(r,0)\bigr)-p_\varepsilon^{t^*}|\geqslant\eta,
	\end{equation}
	for every $n\in\mathbb{N}$ big enough. On the other hand, from \eqref{15} and the fact that $P_\varepsilon$ and $P_\varepsilon^{t^*}$ are topologically conjugated (see Proposition~\ref{P5}), we have that
		\[
		\lim\limits_{n\to+\infty}\Phi_\varepsilon\bigl(nT;\Phi_\varepsilon(t^*;r,0)\bigr)=p_\varepsilon^{t^*}.
		\]
	In particular, for sufficiently large $n\in\mathbb{N}$, we have
		\[
		\bigl|\Phi_\varepsilon\bigl(t^*;(P_\varepsilon)^{k_n}(r,0)\bigr)-p_\varepsilon^{t^*}\bigr|<\eta,
		\]
	contradicting \eqref{22}. The proof for the unstable sets $W^u(p_\varepsilon)$ and $W^u(\gamma_\varepsilon)$ follows similarly.
\end{proof}

\begin{remark}
Since $W^{s,u}(\gamma_\varepsilon)$ intersect $\Omega$ only in crossing points, it follows that they are given by the continuation of $W_{loc}^{s,u}(\gamma_\varepsilon)$ through the global solution of $\mathcal{Z}_\varepsilon$. Hence, we can write
\[W^s(\gamma_\varepsilon)=\bigcup_{\tau\leqslant0}\Phi_\varepsilon^\tau\bigl(W_{loc}^s(\gamma_\varepsilon)\bigr), \quad W^u(\gamma_\varepsilon)=\bigcup_{\tau\geqslant0}\Phi_\varepsilon^\tau\bigl(W_{loc}^u(\gamma_\varepsilon)\bigr).\]
This, in addition with Propositions~\ref{P2} and~\ref{P3} and the fact that $W_{loc}^{s,u}(\gamma_\varepsilon)$ are embedded sub-manifolds that do not intersect $\Omega$, ensures that $W^{s,u}(\gamma_\varepsilon)$ are topological manifolds that fail to be smooth only perhaps at the intersections with $\Omega$. From Proposition~\ref{P1} it follows that, for each $\tau\leqslant0$, $\Phi_\varepsilon^\tau\bigl(W_{loc}^s(\gamma_\varepsilon)\bigr)$ intersects $\Omega$ at most in a finite amount of curves. By making $\tau\to-\infty$ we obtain that $W^{s}(\gamma_\varepsilon)$ intersects $\Omega$ at most in a countable amount of curves. A similar conclusion holds for $W^u(\gamma_{\varepsilon})$.

Furthermore, taking into account the stable and unstable sets $W^{s,u}(p_\varepsilon)$, Proposition~\ref{P6} yields an analogous statement, where the intersection with the switching set $\Sigma$ consists of isolated points rather than curves.
\end{remark}

\section{Proofs of Theorem~\ref{Main1} and Proposition~\ref{MainCoro}}\label{Sec4}

Before we begin with the proof of Theorem~\ref{Main1}, we recall that $\Delta\subset W^u(p_{\varepsilon_0}^{t_0})\cap U_{\varepsilon_0}^{t_0}$ is a $u$-disk intersecting $W^s(p_{\varepsilon_0}^{t_0})\cap U_{\varepsilon_0}^{t_0}$ transversely at some point $q_{\varepsilon_0}^{t_0}\in U_{\varepsilon_0}^{t_0}\setminus\Omega^{t_0}$, and $D\subset W^u(p_{\varepsilon_0}^{t_0})\cap U_{\varepsilon_0}^{t_0}$ is any $u$-disk (see Figure~\ref{Fig12}).

Since $\varepsilon_0\in(0,\varepsilon^*)$ and $t_0\in\mathbb{S}^1$ are fixed throughout the entire proof, we shall omit them for simplicity. We will drop the intersection with $U_{\varepsilon_0}^{t_0}$ in the notation, as well. Moreover, since $\Omega^{t_0}$ is essentially a copy of $\Sigma$ at the layer $t=t_0$, we also write $\Sigma$ instead of $\Omega^{t_0}$. Therefore, we make the identifications $P:=P_{\varepsilon_0}^{t_0}$, $p:=p_{\varepsilon_0}^{t_0}$, $q:=q_{\varepsilon_0}^{t_0}$, and we simply write that $\Delta\subset W^u(p)$ is a $u$-disk intersecting $W^s(p)$ transversely at a point $q\not\in\Sigma$, and $D\subset W^u(p)$ is any $u$-disk.

Finally, we summarize the structure of the proof of Theorem~\ref{Main1} in a few words. First, we use of the results of Section~\ref{Sec3.1} to reduce the problem to a local one, in a neighborhood of the hyperbolic saddle $p$. This reduction is particularly useful when obtaining the proof of other kind of entanglements, such as the heteroclinic one, which is going to be addressed in Section~\ref{Sec5}. Then, we proof the result for the particular case in which $D\subset W^u(p)$ is a sufficiently small neighborhood of $p$ such that $D\cap\Sigma=\emptyset$. From this case we shall see at the end of the proof that the general case, in which $D\subset W^u(p)$ is any $u$-disk, will follow.

\subsection{Proof of Theorem~\ref{Main1}}

Since $p$ is a fixed point of $P$, it follows that there is a neighborhood $V\subset\mathbb{R}^n\setminus\Sigma$ of $p$ such that $P(V)\cap\Sigma=\emptyset$. Restricting $V$ if necessary, we can apply the Stable Manifold Theorem and take coordinates $(x_s,x_u)\in\mathbb{R}^s\times\mathbb{R}^u$ such that $p$ is the origin, and $W^s(p)\cap V$ and $W^u(p)\cap V$ are given by $x_u=0$ and $x_s=0$, respectively (see Figure~\ref{Fig3}).
\begin{figure}[ht]
	\begin{center}
		\begin{overpic}[width=8cm]{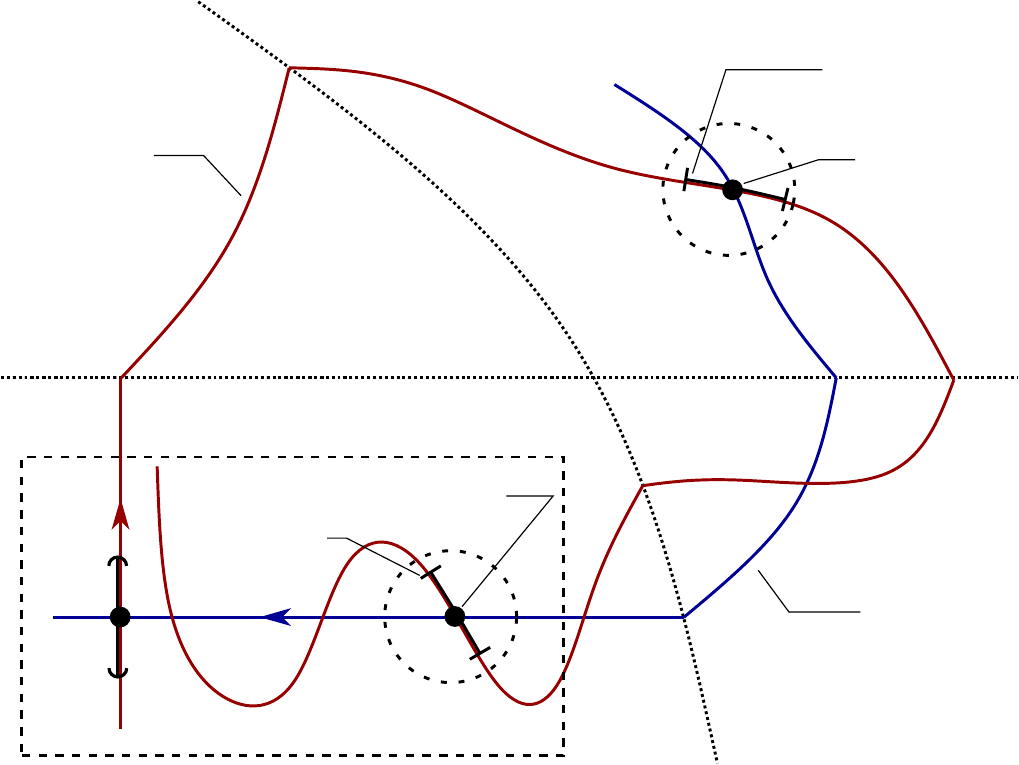} 
			\put(85,58.5){$q$}
			\put(62,49){$B$}
			\put(81,66.5){$\Delta$}
			\put(51,39){$\Sigma$}
			\put(33,3){$P^N(B)$}
			\put(36,25){$P^N(q)$}
			\put(16.5,20.75){$P^N(\Delta)$}
			\put(8,11){$p$}
			\put(6,18.5){$D$}
			\put(56.5,0){$V$}
			\put(1,58.5){$W^u(p)$}
			\put(85,13.5){$W^s(p)$}
		\end{overpic}
	\end{center}
	\caption{Illustration of the set $V$. The stable (resp. unstable) set is represented in blue (resp. red). Colors available in the online version.}\label{Fig3}
\end{figure}
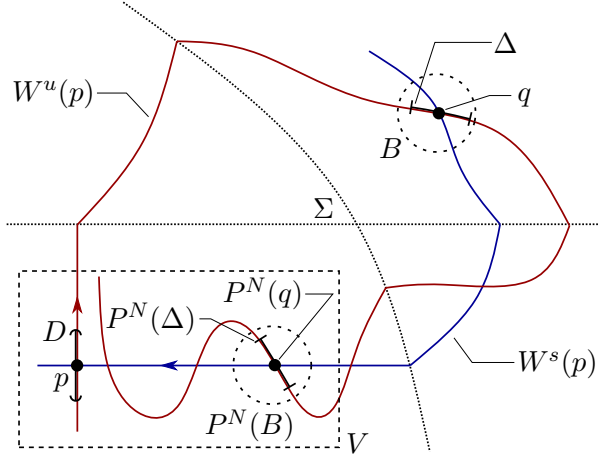
Since $q\in W^s(p)$, it follows that $P^n(q)\to p$ as $n\to\infty$. Thus, there exists $N\in\mathbb{N}$ such that $P^{N}(q)\in V$. In particular, $P^{N}(q)\not\in\Sigma$. Since $q\notin\Sigma$ as well, it follows from Proposition~\ref{P2} that $P^{N}$ is a $\mathcal{C}^\infty$-diffeomorphism in a neighborhood $B\subset\mathbb{R}^n\setminus\Sigma$ of $q$. Therefore, by restricting $\Delta$ if necessary, we have that $P^{N}(\Delta)\subset W^u(p)\setminus\Sigma$ is a $u$-disk intersecting $W^s(p)$ transversely at the point $P^{N}(q)$. Hence, we may assume, without lost of generality, that $q\in V$ and $\Delta\subset V$ (see Figure~\ref{Fig3}). 
	
As previously mentioned, we first consider the case in which $D\subset W^u(p)\setminus\Sigma$ is contained in a sufficiently small neighborhood of $p$. In particular, by restricting $D$ if necessary, we assume $D\subset V$.
	
Consider $\eta>0$. We shall prove that there exists $n\in\mathbb{N}$ big enough such that $P^n(\Delta)$ is contained in the $\eta$-neighborhood of $D$, in relation to the $\mathcal{C}^1$-topology.
	
Taking into account the coordinates $(x_s,x_u)$ in $V$, the map $P\colon V\to\mathbb{R}^n$ can be written as
	\[P(x_s,x_u)=(A^sx_s+P_s(x_s,x_u),A^ux_u+P_u(x_s,x_u)),\]
with $|A^s|\leqslant a<1$, $|(A^u)^{-1}|\leqslant a<1$, $P_s(0,x_u)\equiv0$, $P_u(x_s,0)\equiv0$ and 
\begin{equation}\label{4}
	\frac{\partial P_s}{\partial x_s}(0,0)=0, \quad \frac{\partial P_s}{\partial x_u}(0,0)=0, \quad \frac{\partial P_u}{\partial x_s}(0,0)=0, \quad \frac{\partial P_u}{\partial x_u}(0,0)=0.
\end{equation}
Let $\overline{V}$ be the topological closure of $V$ and consider 
	\[k:=\max\left\{\left|\frac{\partial P_i}{\partial x_j}(r_s,r_u)\right|\colon (r_s,r_u)\in\overline{V} \text{ and } \, i,j\in\{s,u\}\right\}.\]
From \eqref{4}, by restricting  $V$ if necessary, we can assume that $k\geqslant0$ is small enough such that
\begin{equation}\label{5}
	a_1:=a+k<1, \quad b:=a^{-1}-k>1, \quad k<\frac{(b-1)^2}{4}, \quad k<1.
\end{equation}
Now, let $v_0=(v_0^s,v_0^u)\in\mathbb{R}^s\times\mathbb{R}^u$ be a tangent to $\Delta$ at $q$. Since $\Delta$ is transversal to $W^s(p)$, it follows that $v_0^u\neq0$, which allows us to define the inclination 
	\[ 
		\lambda_0:=\frac{|v_0^s|}{|v_0^u|}. 
	\]
Let $n\geqslant1$ and
	\[q_n=P^n(q), \quad v_n=(v_n^s,v_n^u)=DP(q_{n-1})v_{n-1}, \quad \lambda_n=\frac{|v_n^s|}{|v_n^u|}.\]
		
We claim that there exists $n_0\in\mathbb{N}$, independent of $v_0$, such that	
\begin{equation}\label{7}
	\lambda_n\leqslant\frac{b-1}{4},
\end{equation}
for every $n\geqslant n_0$.	
		
Indeed, let $v^*$ be a unitary vector tangent to $\Delta$ at $q$ with maximal inclination $\lambda^*$. In order to simplify the notation, let $P_{i,j}^n=\frac{\partial P_i}{\partial x_j}(q_{n-1})$. We  notice that $P_{u,s}^n=0$, since $P_u(x_s,0)\equiv0$. Therefore, from \eqref{5}, we have that
\begin{equation}\label{9}
	\begin{array}{rl}
		\displaystyle \lambda_n &\displaystyle= \frac{|v_n^s|}{|v_n^u|}=\frac{|(A^s+P_{s,s}^n)v_{n-1}^s+P_{s,u}^nv_{n-1}^u|}{|(A^u+P_{u,u}^n)v_{n-1}^u|} \vspace{0.2cm} \\
		
		&\displaystyle \leqslant\frac{(a+k)|v_{n-1}^s|+k|v_{n-1}^u|}{(a^{-1}-k)|v_{n-1}^u|} =\frac{(a+k)\lambda_{n-1}+k}{a^{-1}-k} \vspace{0.2cm} \\
			
		&\displaystyle= \frac{a_1\lambda_{n-1}+k}{b}<\frac{1}{b}\lambda_{n-1}+k\frac{1}{b}.
	\end{array}
\end{equation}
By recursively applying~\eqref{9}, in addition with properties of geometric series and \eqref{5}, we get
\begin{equation}\label{6}
	\lambda_n\leqslant\frac{\lambda_0}{b^n}+k\sum_{i=1}^{n}\frac{1}{b^i}<\frac{\lambda_0}{b^n}+\frac{k}{b-1}\leqslant\frac{\lambda^*}{b^n}+\frac{k}{b-1} .
\end{equation}
The fact that $\lambda^*/b^n\to0$ and $k/(b-1)<(b-1)/4$ concludes the claim.
	
Consider now $V_1=(-\delta,\delta)^s\times D$ with $\delta>0$ small enough such that $\delta<\eta$, $V_1\subset V$, and
\vspace{0.2cm}
\begin{equation}\label{25}
	k_1:=\max\left\{\left|\frac{\partial P_i}{\partial x_j}(r_s,r_u)\right|\colon (r_s,r_u)\in\overline{V_1} \text{ and } \, i,j\in\{s,u\}\right\}<\min\{\eta,k\}.
\end{equation}
Let $n_1\geqslant n_0$ be large enough so that $q_{n_1}\in V_1$. From \eqref{7}, it follows that there exists a sufficiently small $u$-neighborhood $\Delta_0$ of $q_{n_1}$, with $\overline{\Delta_0}\subset P^{n_1}(\Delta)\cap V_1$, such that for any $r\in\Delta_0$ and for any tangent vector $w_0=(w_0^s,w_0^u)$ of $\Delta_0$ at $r$, we have
\begin{equation}\label{8}
	\mu_0:=\frac{|w_0^s|}{|w_0^u|}\leqslant\frac{b-1}{2}.
\end{equation}
For each $n\in\mathbb{N}$, let $\Delta_n\subset\Delta_0$ be a sufficiently small $u$-neighborhood of $q_{n_1}$ such that $P^k(\Delta_n)\subset V_1$ for $k\in\{1,\dots,n\}$ (from Remark~\ref{R1}, we recall that we may not have control over all the infinitely many iterates of $\Delta_0$). Given $n\in\mathbb{N}$, $r\in\Delta_n$, and a vector $(w_0^s,w_0^u)\in\mathbb{R}^s\times\mathbb{R}^u$ tangent to $\Delta_n$ at $r$, consider
	\[r_k=P^k(r), \quad w_k=(w_k^s,w_k^u)=DP(r_{k-1})w_{k-1}, \quad \mu_k=\frac{|w_k^s|}{|w_k^u|},\]
for $k\in\{0,\dots,n\}$ (see Figure~\ref{Fig4}).
\begin{figure}[ht]
	\begin{center}
		\begin{overpic}[width=9cm]{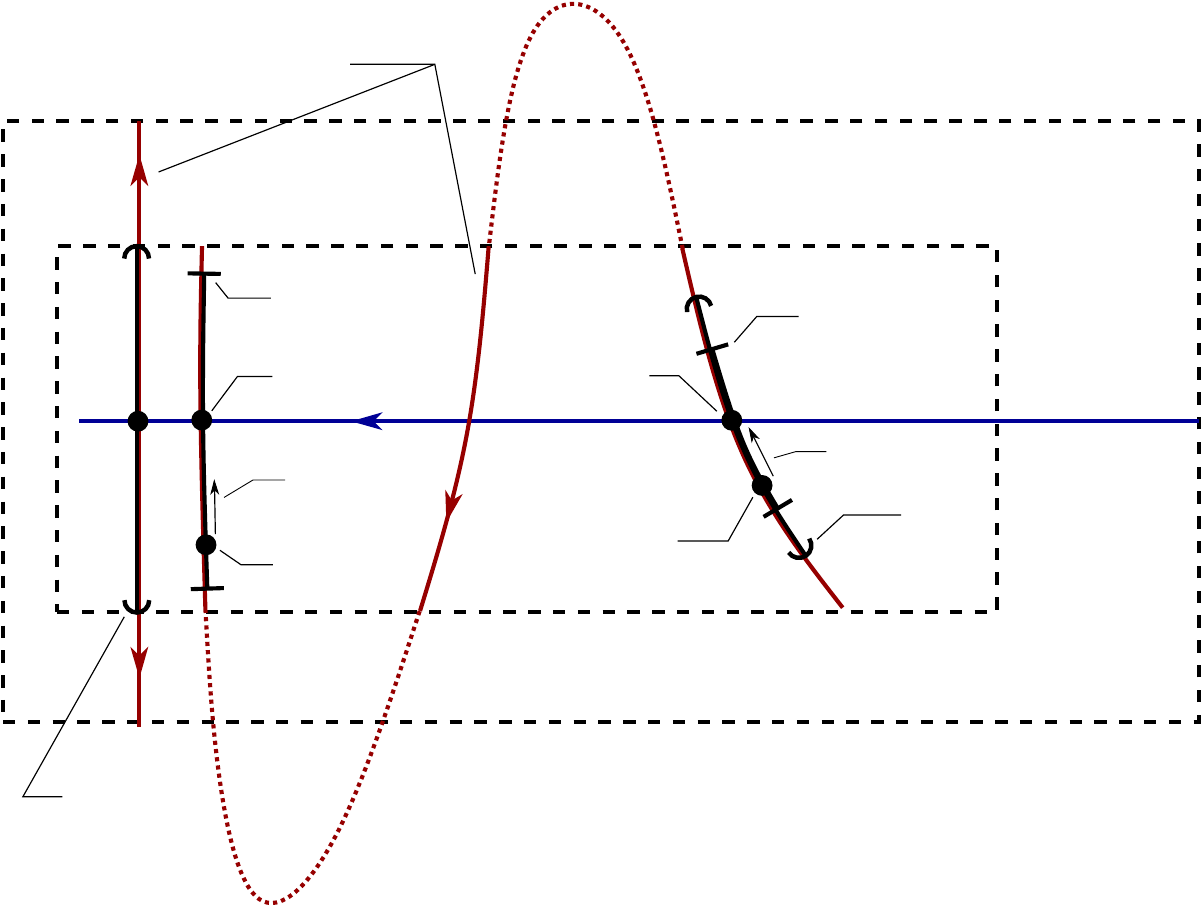} 
			\put(101,13){$V$}
			\put(85,23){$V_1$}
			\put(86,42){$W^s(p)$}
			\put(16,69){$W^u(p)$}
			\put(67,47.5){$\Delta_n$}
			\put(75.5,31){$\Delta_0$}
			\put(69.5,36.75){$w_0$}
			\put(53.5,29.5){$r$}
			\put(48,43.5){$q_{n_1}$}
			\put(23.5,27.5){$r_n$}
			\put(24.5,34.5){$w_n$}
			\put(24,49){$P^n(\Delta_n)$}
			\put(23.5,43){$q_{n+n_1}$}
			\put(5.5,7.5){$D$}
		\end{overpic}
	\end{center}
	\caption{Illustration of the sets $\Delta_n$ and $P^n(\Delta_n)$. The stable (resp. unstable) set is represented in blue (resp. red). Colors available in the online version.}\label{Fig4}
\end{figure}

Similarly to \eqref{9}, let us rewrite $P_{i,j}^n=\frac{\partial P_i}{\partial x_j}(r_{n-1})$, and observe that, from \eqref{5} and \eqref{8}, we have
\begin{equation}\label{10}
	\begin{array}{ll}
		\displaystyle \mu_1 &\displaystyle= \frac{|w_1^s|}{|w_1^u|}=\frac{|(A^s+P_{s,s}^n)w_{0}^s+P_{s,u}^nw_{0}^u|}{|P_{u,s}^nw_{0}^s+(A^u+P_{u,u}^n)w_{0}^u|} \vspace{0.2cm} \\
		
		&\displaystyle \leqslant\frac{(a+k)|w_{0}^s|+k_1|w_{0}^u|}{(a^{-1}-k)|w_{0}^u|-k|w_{0}^s|} = \frac{(a+k)\mu_{0}+k_1}{(a^{-1}-k)-k\mu_{0}} \vspace{0.2cm} \\
		
		&\displaystyle <\frac{\mu_{0}+k_1}{b-\frac{1}{2}k(b-1)}<\frac{\mu_0+k_1}{b-\frac{1}{2}(b-1)}=\frac{\mu_0+k_1}{\frac{1}{2}(b+1)}.
	\end{array}
\end{equation}
Taking $b_1:=\frac{1}{2}(b+1)>1$ and relation \eqref{10} into account, we notice that
\begin{equation}\label{24}
	\mu_1<\frac{\mu_0}{b_1}+k_1\frac{1}{b_1}.
\end{equation}
From \eqref{5}, \eqref{25}, \eqref{8}, and \eqref{24}, we obtain
\begin{equation}\label{26}
	\begin{array}{ll}
		\displaystyle \mu_1 &\displaystyle< \frac{1}{b_1}(\mu_0+k_1)<\frac{1}{b_1}\left(\frac{b-1}{2}+k\right)<\frac{1}{b_1}\left(\frac{b-1}{2}+\left(\frac{b-1}{2}\right)^2\right) \vspace{0.2cm} \\
		&\displaystyle=\frac{1}{b_1}\frac{b-1}{2}\left(1+\frac{b-1}{2}\right)=\frac{1}{b_1}\frac{b-1}{2}b_1=\frac{b-1}{2}.
	\end{array}
\end{equation}	
That is, $\mu_1$ also satisfies \eqref{8}. By recursively applying \eqref{10}, we notice that 
\begin{equation}\label{27}
	\mu_n\leqslant\frac{\mu_0}{b_1^n}+\frac{k_1}{b_1-1},
\end{equation}
similar to \eqref{6}. By taking $w^*$ unitary and with maximal inclination $\mu^*$, among all the points $r\in\overline{\Delta_0}$, we obtain from~\eqref{27} that
\[
	\mu_n\leqslant\frac{\mu_0}{b_1^n}+\frac{k_1}{b_1-1}\leqslant\frac{\mu^*}{b_1^n}+\frac{k_1}{b_1-1}.
\]
Hence, we conclude that there exists $n_2\in\mathbb{N}$, depending neither on $r$ nor on $w_0$, such that 
\begin{equation}\label{11}
	\mu_n\leqslant\left(1+\frac{1}{b_1-1}\right)\eta,
\end{equation}
for every $n\geqslant n_2$. 

From \eqref{11} and the fact that $V_1$ itself lies in the $\eta$-neighborhood of $D$ in relation to the $\mathcal{C}^0$-topology, we conclude that $P^{n_2}(\Delta_{n_2})$ lies in the $\eta$-neighborhood of $D$ in relation to the $\mathcal{C}^1$-topology. This proves the theorem for the case in which $D\subset W^u(p)\cap V$ is a sufficiently small $u$-neighborhood of $p$.
	
The general case, where $D\subset W^u(p)$ is any $u$-neighborhood, is as follows. Let $D\subset W^u(p)$ be a $u$-disk and observe that the iterates $P^{-m}(D)$ get arbitrarily close to $p$. If $D\cap\Sigma=\emptyset$, then, from Proposition~\ref{P2}, we have that $P^{-m}$ is a $\mathcal{C}^\infty$-diffeomorphism in a neighborhood of $D$ for $m\in\mathbb{N}$ big enough. From the previous case we can take $P^n(\Delta_n)$ arbitrarily $\mathcal{C}^1$-close to $P^{-m}(D)$ and thus $P^{n+m}(\Delta_n)$ contain a $u$-disk arbitrarily $\mathcal{C}^1$-close to $D$. If $D\cap\Sigma\neq\emptyset$, then the same reasoning applies, but restricted to the $\mathcal{C}^0$-topology due to Proposition~\ref{P3}. This concludes the proof of Theorem~\ref{Main1}.

\subsection{Consequences of Theorem~\ref{Main1}}

We now present two consequences of Theorem~\ref{Main1}. In the first, discussed in Remark~\ref{Remark2}, we show that the iterates of $\Delta_n$ (see Figure~\ref{Fig4}) exhibit a stretching behavior even as their inclination tends to zero (see Figure~\ref{Fig10}).
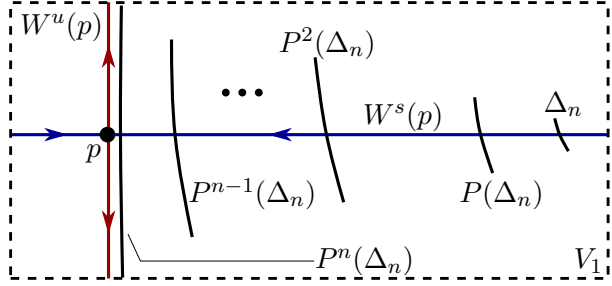
\begin{figure}[ht]
	\begin{center}
		\begin{overpic}[width=8cm]{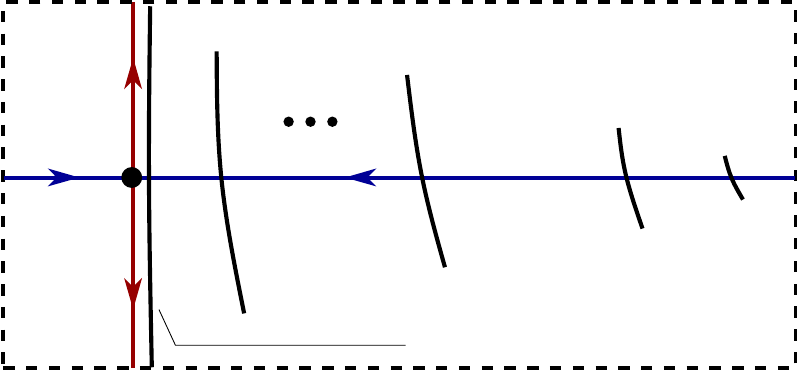} 
			\put(94,2){$V_1$}
			\put(59,26){$W^s(p)$}
			\put(2,41){$W^u(p)$}
			\put(13,20.5){$p$}
			\put(89,28){$\Delta_n$}
			\put(75,13){$P(\Delta_n)$}
			\put(45,38){$P^2(\Delta_n)$}
			\put(30,13){$P^{n-1}(\Delta_n)$}
			\put(51,2){$P^n(\Delta_n)$}
		\end{overpic}
	\end{center}
	\caption{Illustration of the stretching of $\Delta_n$. The stable (resp. unstable) set is represented in blue (resp. red). Colors available in the online version.}\label{Fig10}
\end{figure}
In the sequel, we prove Proposition~\ref{MainCoro}, which establishes in the entanglement between $W^{s}(\gamma_{\varepsilon_0})$ and $W^{u}(\gamma_{\varepsilon_0})$ (see Figure~\ref{Fig13}).

\begin{remark}\label{Remark2}
	Similarly to \eqref{10}, we notice that
		\[\begin{array}{l}
			\displaystyle |w_{n+1}^u|=|(A^u+P_{u,u}^n)w_n^u+P_{u,s}^nw_n^s|\geqslant |(A^u+P_{u,u}^n)w_n^u|-|P_{u,s}^nw_n^s| \vspace{0.2cm} \\
			\displaystyle \qquad\qquad\geqslant a^{-1}|w_n^u|-|P_{u,u}^nw_n^u|-|P_{u,s}^nw_n^s|\geqslant a^{-1}|w_n^u|-k(|w_n^u|+|w_n^s|).
		\end{array}\]
	Therefore we have that
		\[\frac{|w_{n+1}^u|}{|w_n^u|}\geqslant a^{-1}-k(1+\mu_n)=b-k\mu_n>1,\]
	provided $V$ is sufficiently small and $n$ sufficiently large. Hence the inclination $\mu_n$ converges to zero at the same time that $|w_n|$ diverges, growing at a ratio that converges to $b>1$. 
\end{remark}

\begin{proof}[Proof of Proposition~\ref{MainCoro}]
	Let $t_0\in\mathbb{S}^1$ be such that $W^u(p_{\varepsilon_0}^{t_0})$ and $W^s(p_{\varepsilon_0}^{t_0})$ intersect transversely at some point $q_{\varepsilon_0}^{t_0}\in U^{t_0}_{\varepsilon_0}\setminus\Omega^{t_0}$. Given $t\in\mathbb{S}^1$, it follows from Proposition~\ref{P5} that $W^u(p_{\varepsilon_0}^{t})$ and $W^s(p_{\varepsilon_0}^{t})$ intersects at the point $q_{\varepsilon_0}^t:=\Phi_{\varepsilon}(t-t_0;q_{\varepsilon_0}^{t_0},t_0)$. Since $q_{\varepsilon_0}^{t_0}\not\in\Omega^{t_0}$, from Proposition~\ref{P2} we know that the conjugation map $\Phi_\varepsilon^{t-t_0}$ is a diffeomorphism in a neighborhood of $q_{\varepsilon_0}^{t_0}$ except perhaps when $q_{\varepsilon_0}^t\in\Omega^t$. On the other hand, Proposition~\ref{P1} asserts that there are at most finitely many $\{t_1,\dots,t_k\}$ such that $q_{\varepsilon_0}^{t_k}\in\Omega^{t_k}$. The result now follows from the fact that transversal intersections are preserved by diffeomorphisms.
\end{proof}

\section{Further thoughts}\label{Sec5}

As previously discussed in Section~\ref{Sec2}, Theorem~\ref{Main1} holds in a more general context, which is detailed in the sequel. 

First let $Z$ be as in Section~\ref{Sec2}, i.e. an autonomous smooth vector field defined at $\mathbb{R}^n$ with switching set $\Sigma$. However, instead of a homoclinic connection $\Gamma\subset W^s(p)\cap W^u(p)$ (see Figure~\ref{Fig8}), we assume the existence of a heteroclinic connection $\Gamma\subset W^s(p)\cap W^u(r)$ intersecting $\Sigma$ only in crossing points, with $r\in\mathbb{R}^n\setminus\Sigma$ being a hyperbolic saddle distinct from $p\in\mathbb{R}^n\setminus\Sigma$. 

Consider now a non-autonomous perturbation $Z_\varepsilon$ of $Z$, similarly to~\eqref{28}, and the autonomous vector field $\mathcal{Z}_\varepsilon$ obtained by adopting $\dot t=1$, similarly to~\eqref{21}. Analogously to the construction made in Section~\ref{Sec2}, there is $\varepsilon^*>0$ small enough such that for each $0<\varepsilon<\varepsilon^*$, there exist a neighborhood $\mathcal{U}_{\varepsilon}$ of $\Gamma\times\mathbb{S}^1$ such that for each $t_0\in\mathbb{S}^1$ we have a well defined Poincaré map $P_\varepsilon^{t_0}\colon U_\varepsilon^{t_0}\to V_\varepsilon^{t_0}:=P_\varepsilon^{t_0}(U_\varepsilon^{t_0})$, associated with $\mathcal{Z}_\varepsilon$, and given by
\begin{equation}\label{29}
	P_\varepsilon^{t_0}(x)=\Phi_\varepsilon(T;x,t_0),
\end{equation}
where $U_\varepsilon^{t_0}:=\mathcal{U}_\varepsilon\cap\{t=t_0\}$ and $\Phi_\varepsilon(\tau;x,t)$ is the solution of $\mathcal{Z}_\varepsilon$ with initial condition $\Phi_\varepsilon(0;x,t)=(x,t)$. As a consequence, it follows that $W^s(p_{\varepsilon_0}^{t_0})\cap\mathcal{U}_{\varepsilon_0}$ and $W^u(r_{\varepsilon_0}^{t_0})\cap\mathcal{U}_{\varepsilon_0}$ also get entangled, similarly to what is represented in Figure~\ref{Fig7}$(b)$. The proof of this fact follows similarly to the proof of Theorem~\ref{Main1}, i.e., we first reduce the problem to a local one in a neighborhood of $p$ and then we repeat the proof. This leads us to the following result.

\begin{main}[$\lambda$-lemma for heteroclinic tangles of PSVF]\label{Main3}
	Let $P_\varepsilon^{t}$ denote the time-$T$-map defined in~\eqref{29}, and let $p_{\varepsilon}^{t}$ and $r_\varepsilon^t$ be two hyperbolic saddles. Suppose that for some $0<\varepsilon_0<\varepsilon^*$ and $t_0\in\mathbb{S}^1$ there exists a $u$-disk $\Delta\subset W^u(r_{\varepsilon_0}^{t_0})\cap U^{t_0}_{\varepsilon_0}$ intersecting $W^s(p_{\varepsilon_0}^{t_0})\cap U^{t_0}_{\varepsilon_0}$ transversely at some point $q_{\varepsilon_0}^{t_0}\in U^{t_0}_{\varepsilon_0}\setminus\Omega^{t_0}$. Then for every $u$-disk $D\subset W^u(p_{\varepsilon_0}^{t_0})\cap U^{t_0}_{\varepsilon_0}$ we have that 
		\[\bigcup_{n=0}^{\infty}(P_{\varepsilon_0}^{t_0})^n\bigl(\Delta\cap U^{t_0}_{\varepsilon_0}\bigr)\]
	contains $u$-disks arbitrarily close in the $\mathcal{C}^1$-topology (resp. in the $\mathcal{C}^0$-topology) to $D$ when $D\cap\Omega^{t_0}=\emptyset$ (resp. when $D\cap\Omega^{t_0}\neq\emptyset$).
\end{main}

In particular, similarly to Proposition~\ref{MainCoro}, there is also an entanglement between the stable and unstable sets $W^s(\gamma^p_\varepsilon)$ and $W^u(\gamma^r_\varepsilon)$, in relation to $\mathcal{Z}_\varepsilon$, of the periodic orbits $\gamma^p_\varepsilon$ and $\gamma^r_\varepsilon$ related to $p_\varepsilon$ and $r_\varepsilon$, similarly to what is represented Figure~\ref{Fig13}. 

In the non-perturbative context and for the sake of completeness, let now $Z=Z(x,t)$ be itself a non-autonomous piecewise smooth vector field with switching set $\Sigma$, $T$-periodic in $t$, $T>0$, and with a $T$-periodic hyperbolic orbit $\sigma\subset\mathbb{R}^n\setminus\Sigma$ of saddle type (i.e. $W^{s,u}(\sigma)\setminus\sigma\neq\emptyset$). Let $\mathcal{Z}$ be the autonomous piecewise smooth vector field obtained in the extended phase space of $Z$ and with $\Omega:=\Sigma\times\mathbb{S}^1$ as switching set, i.e. $\mathcal{Z}$ is given by
	\[\dot x= Z(x,t), \quad \dot t=1,\]
where $(x,t)\in\mathbb{R}^n\times\mathbb{R}/T\mathbb{Z}=\mathbb{R}^n\times\mathbb{S}^1$.
Let also $\Phi(\tau;x,t)$ denote the solution of $\mathcal{Z}$ with initial condition $\Phi(0;x,t)=(x,t)$. 

Let $\gamma$ be the periodic orbit of $\mathcal{Z}$ obtained by considering $\sigma$ in the extended phase space. Consider a point $(q,t_0)\in W^s(\gamma)\setminus\gamma$ and suppose that $\{\Phi(\tau;q,t_0)\colon\tau\geqslant0\}$ intersects $\Omega$ only at crossing points. It follows from Section~\ref{Sec3.1} that for each $t\in\mathbb{S}^1$ we have a well defined Poincaré map 
\begin{equation}\label{30}
	P^t(x)=\Phi(T;x,t),
\end{equation}
defined in a neighborhood of $q^{t_0+t}:=\Phi(t;q,t_0)$, relatively to $\mathbb{R}^n\times\{t_0+t\}$, and satisfying the same properties stated at Section~\ref{Sec3.2}. In particular, for any two $t_1$, $t_2\in\mathbb{S}^1$ we have that the two respective Poincar\'e maps are conjugated and $q^{t_0+t}\not\in\Omega$ for all $t\in\mathbb{S}^1$, except possibly by finitely many.

For simplicity from now on we consider only the Poincar\'e map $P:=P^0$ given by $t_0+t=0$, set the hyperbolic saddle $p:=\gamma\times\{0\}$, denote $q:=q^0$ and suppose without loss of generality that $q\not\in\Omega^{0}$, where we recall that $\Omega^0=\Omega\cap\{t=0\}$. Notice also that $p\not\in\Omega^0$, since $\sigma\subset\mathbb{R}^n\setminus\Sigma$.

Therefore, if $W$ is an invariant set of $P$ having a $u$-disk $\Delta$ intersecting $W^s(p)$ transversely at $q$, then it follows analogously to Theorem~\ref{Main1} that its iterates accumulates at $W^{u}(p)$, similarly to Figure~\ref{Fig6}. This leads us to the following result.

\begin{main}[$\lambda$-lemma for non-autonomous PSVF]\label{Main4}
	Let $P^{t}$ denote the time-$T$-map defined in~\eqref{30}, and let $p^{t}$ and a hyperbolic saddle. Suppose that for some $t_0\in\mathbb{S}^1$ there exists a $u$-disk $\Delta\subset U^{t_0}$ intersecting $W^s(p^{t_0})\cap U^{t_0}$ transversely at some point $q^{t_0}\in U^{t_0}\setminus\Omega^{t_0}$. Then for every $u$-disk $D\subset W^u(p^{t_0})\cap U^{t_0}$ we have that 
		\[\bigcup_{n=0}^{\infty}(P^{t_0})^n(\Delta)\]
	contains $u$-disks arbitrarily close in the $\mathcal{C}^1$-topology (resp. in the $\mathcal{C}^0$-topology) to $D$ when $D\cap\Omega^{t_0}=\emptyset$ (resp. when $D\cap\Omega^{t_0}\neq\emptyset$).
\end{main}

In particular if $\Delta\subset W^u(p)$ (resp. $\Delta\subset W^u(r)$ for some other hyperbolic saddle $r$), then we have an \emph{homoclinic} (resp. \emph{heteroclinic}) \emph{tangle}, but this time without the need to take a perturbation $\mathcal{Z}_\varepsilon$ of $\mathcal{Z}$. Moreover, similarly to Proposition~\ref{MainCoro}, the accumulation provided by Theorem~\ref{Main4} also occurs in the extended phase space, similarly to Figure~\ref{Fig13}. 

Observe that by reversing the independent time variable $\tau$, a result similar to Theorem~\ref{Main4} also holds if we consider $(q,t_0)\in W^u(\gamma)\setminus\gamma$ instead.

Suppose now that for some $t_0\in\mathbb{S}^1$ we have two hyperbolic saddles $p$, $r\in\mathbb{R}^n\setminus\Omega^{t_0}$ of the Poincaré map $P^{t_0}$, obtained from two different hyperbolic periodic orbits of saddle type $\gamma^p$ and $\gamma^r$ of $\mathcal{Z}$. We say that $p\sim r$ if $W^u(p)$ and $W^s(r)$ have a transversal intersection in $\mathbb{R}^n\setminus\Omega^{t_0}$. In the next result we prove that the relation $\sim$ is transitive (see Figure~\ref{Fig11}).
\begin{figure}[ht]
	\begin{center}
		\begin{overpic}[width=8cm]{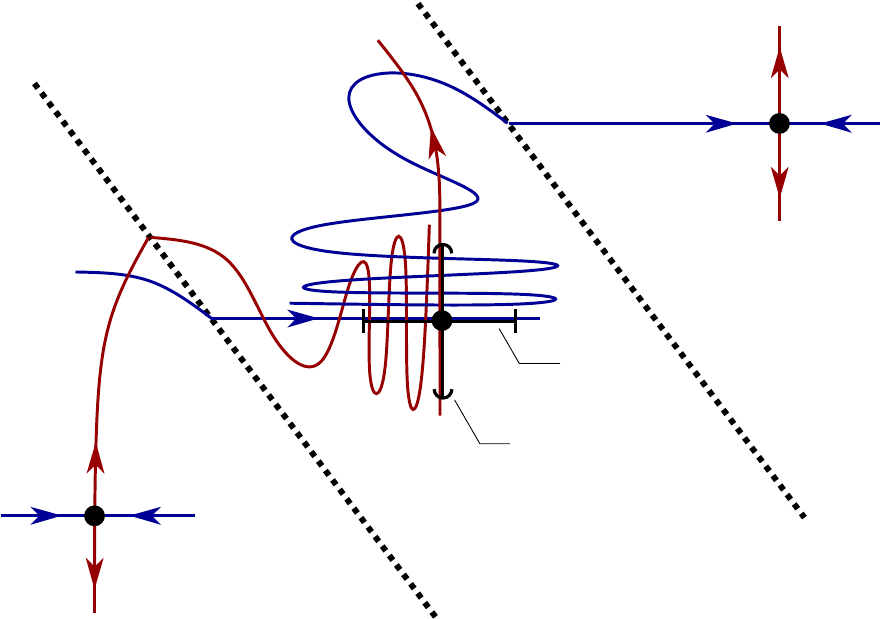} 
			\put(12,8){$p_1$}
			\put(51.5,30.5){$p_2$}
			\put(89.5,52.5){$p_3$}
			\put(64,27){$D^s$}
			\put(58,18){$D^u$}
			\put(50,-1){$\Sigma$}
			\put(92,10){$\Sigma$}
		\end{overpic}
	\end{center}
	\caption{Illustration of Corollary~\ref{Coro1}. The stable (resp. unstable) sets of $p_1$, $p_2$ and $p_3$ is represented in blue (resp. red). Colors available in the online version}\label{Fig11}
\end{figure}
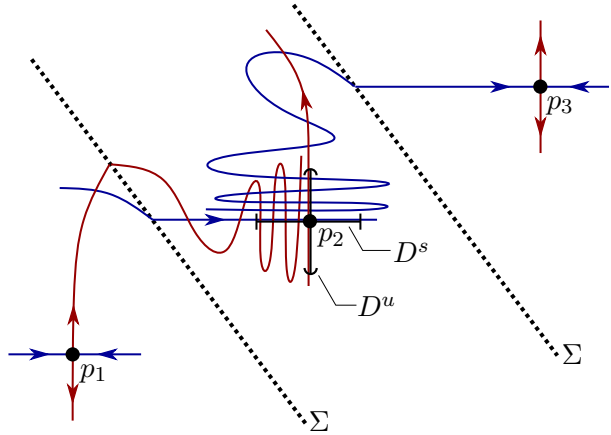 
For simplicity we drop $t_0$ from the notation and make the identification $\Omega^{t_0}\approx\Sigma$.

\begin{proposition}\label{Coro1}
	Let $p_1$, $p_2$ and $p_3\in\mathbb{R}^n\setminus\Sigma$ be three hyperbolic saddles of the Poincaré map $P$, obtained from three periodic hyperbolic orbits of $\mathcal{Z}$ of saddle type. If $p_1\sim p_2$ and $p_2\sim p_3$, then $p_1\sim p_3$.
\end{proposition}

\begin{proof}
	Let $D^s\subset W^s(p_2)\setminus\Sigma$ (resp. $D^u\subset W^u(p_2)\setminus\Sigma$) be a small enough $s$-neighborhood (resp. $u$-neighborhood) of $p_2$. Since $p_1\sim p_2$, we have from Theorem~\ref{Main1} that $W^u(p_1)$ contain $u$-disks arbitrarily $\mathcal{C}^1$-close to $D^u$. Similarly, from the condition $p_2\sim p_3$, we have that $W^s(p_3)$ contains $s$-disks arbitrarily $\mathcal{C}^1$-close to $D^s$. Since $D^u$ and $D^s$ intersects transversely at $p_2$, it follows that $W^u(p_1)$ and $W^s(p_3)$ must also intersect transversely at a point $q$ arbitrarily close to $p_2$ and in particular in $\mathbb{R}^n\setminus\Sigma$.
\end{proof}

It follows from Proposition~\ref{Coro1} that if the Poincaré map $P$ have hyperbolic saddles $p_1,\dots,p_n$ such that $p_i\sim p_{i+1}$, $i\in\{1\dots,n-1\}$, and $p_n\sim p_1$, then $p_i\sim p_j$ for any two $i$, $j\in\{1,\dots,n\}$, included the case $i=j$. That is, if the hyperbolic saddles form a ``circular structure", then we have a \emph{mutual entanglemen} between all of them (see Figure~\ref{Fig14}).
\begin{figure}[ht]
	\begin{center}
		\begin{minipage}{6cm}
			\begin{center} 
				\begin{overpic}[width=5.5cm]{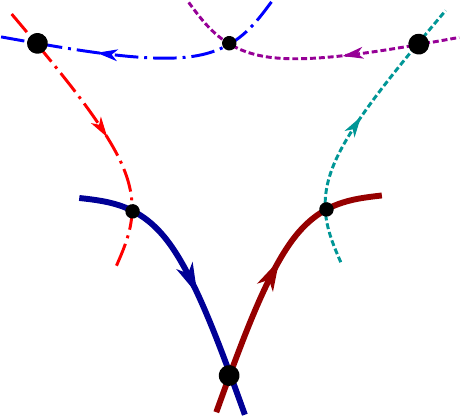} 
					\put(53,7.5){$p_1$}
					\put(89,74){$p_2$}
					\put(4,74){$p_3$}
				\end{overpic}
				
				$(a)$
			\end{center}
		\end{minipage}
		\begin{minipage}{6cm}
			\begin{center} 
				\begin{overpic}[width=5.5cm]{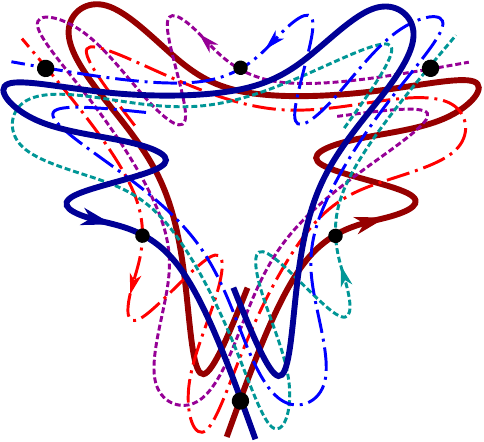} 
				\end{overpic}
				
				$(b)$
			\end{center}
		\end{minipage}
	\end{center}
	\caption{$(a)$ illustrates the circular structure formed by three saddle points and their stable and unstable sets, while $(b)$ shows the mutual entanglement among all parties in the case $n=3$. For simplicity, discontinuities are not depicted in the illustration. Colors available in the online version.}\label{Fig14}
\end{figure}
Moreover, similarly to Proposition~\ref{MainCoro}, this mutual entanglement also occurs in the extended phase space, similarly to Figure~\ref{Fig13}. 

We now observe that the hypothesis of having a piecewise diffeomorphism $P\colon U\to V$ \emph{embedded} in a flow of an autonomous piecewise smooth vector field $\mathcal{Z}$ is essential to reduce the proof of Theorem~\ref{Main1} to a local one. Indeed, the key fact for this claim follows from Proposition~\ref{P2}, which proves that $P^N$ is a diffeomorphism in a neighborhood of $q$, provided $q\not\in\Sigma$ and $P^N(q)\not\in\Sigma$. In particular, $P^N$ is a diffeomorphism even if $P^k(q)\in\Sigma$ for some $k\in\{1,\dots,n-1\}$ and thus it is not necessary to have previous information about the orbit of $q$.

Nevertheless, we can state a \emph{$\lambda$-Lemma for Piecewise Smooth Maps}, provided we add this information on the hypothesis. More precisely, given a smooth manifold $M$ of dimension $m$ and immersed sub-manifolds $\Sigma_1,\dots,\Sigma_k$, each one of codimension at least one, set $\Sigma=\cup_{i=1}^{k}\Sigma_i$ and let $\operatorname{Diff}_{\Sigma}(M)$ be the set of homeomorphisms $F\colon M\to M$ such that $F$ is a smooth diffeomorphism when restricted to $M\setminus\bigl(\Sigma\cup F^{-1}(\Sigma)\bigr)$. Let also $\Lambda:=\cup_{n=0}^{\infty} F^{-n}(\Sigma)$.

\begin{main}\label{Main2}[$\lambda$-lemma for piecewise smooth maps]
	Let $p\in M\setminus\Sigma$ be a hyperbolic saddle of $F\in\operatorname{Diff}_{\Sigma}(M)$ and $D\subset W^u(p)$ a $u$-disk. Let $\Delta$ be a $u$-disk intersecting $W^s(p)$ transversely at some point $q\not\in\Lambda$. Then $\cup_{n=0}^{\infty}P^n(\Delta)$ contains $u$-disks arbitrarily $\mathcal{C}^1$-close (resp. $\mathcal{C}^0$-close) to $D$ if $D\cap\Lambda=\emptyset$ (resp. $D\cap\Lambda\neq\emptyset$).
\end{main}

The proof of Theorem~\ref{Main2} proceeds similarly to that of Theorem~\ref{Main1}. The main difference lies in the fact that we now require the stronger assumption $q\not\in\Lambda$, instead of $q\not\in\Sigma$, in order to reduce the argument to a local one.

Finally, we observe that since $\Lambda$ is a countably union of sub-manifolds, it can be dense in $M$. Moreover, from a probabilistic point of view, even if $M$ is endowed with a measure $\mu$, $F$ is measurable, and $\mu(\Sigma)=0$, it is not clear if $\mu(\Lambda)=0$. See~\cite{Luzin} and the references therein. Thus, the hypothesis of $q\not\in\Lambda$ might be too restrictive. This, in fact, highlights the strength of assuming that the piecewise smooth diffeomorphism is embedded in a flow of a piecewise smooth vector field.

\section{Conclusion and further steps}

It follows from Theorems~\ref{Main1}, \ref{Main3}, and~\ref{Main4} that the homoclinic and heteroclinic tangles of smooth \emph{vector fields} have their clear counterparts in the piecewise smooth framework. However, in case of smooth \emph{maps}, it follows from Theorem~\ref{Main2} and the discussion made at it that its counterpart might need stronger hypothesis. Whether the assumptions made in Theorem~\ref{Main2} are minimal is left as an open problem. 

Moreover it is clear from Proposition~\ref{Coro1} and its consequences, illustrated in Figure~\ref{Fig14}, that these entanglements may result in chaos. Therefore as a next step the authors will push for a piecewise smooth counterpart of the seminal Birkoff-Smale Homoclinic Theorem, specially for vector fields. Similar to the smooth framework, we hope that the results proved in this paper shall be useful in this next step. 

\section*{Acknowledgments}

This work is supported by CNPq, grant 304798/2019-3, by Agence Nationale de la Recherche (ANR), project ANR-23-CE40-0028, S\~ao Paulo Research Foundation (FAPESP), grants 2021/01799-9 and 2023/02959-5, and PROPe-UNESP, project 5482.

\end{document}